\documentclass{article}
\usepackage[utf8]{inputenc}
\usepackage[T1]{fontenc}
\usepackage{geometry}
\usepackage{caption,subcaption}
\usepackage{amsmath,amssymb,amsfonts,amsthm}
\usepackage{indentfirst}
\usepackage{graphicx}
\usepackage{verbatim}
\usepackage{mathrsfs}
\usepackage{mathbbol}
\usepackage[english]{babel}
\usepackage{comment}
\usepackage{authblk} 
\usepackage{tikz}
\usetikzlibrary{arrows}
\usepackage[margin=16pt,font=small,labelfont=bf]{caption} 
\usepackage{float}
\usepackage{parskip}
\usepackage{enumitem}
\usepackage{hyperref}

\newtheorem{thm}{Theorem}[section]
\newtheorem*{thm*}{Theorem}
\newtheorem{proposition}[thm]{Proposition}
\newtheorem*{proposition*}{Proposition}

\newtheorem{lemma}[thm]{Lemma} 
\newtheorem{definition}[thm] {Definition}

\numberwithin{equation}{section} 
\newtheorem{remark}[thm]{Remark}

\title{Emergence of metastability on the hyperbolic lattice: Effects of boundary conditions.}

\author{
Vanessa \textsc{Jacquier}\thanks{ University of Padova, via Luigi Luzzatti 4, 35121, Padova, Italy.\newline $_{}$\hfill    \href{mailto:vanessa.jacquier@unipd.it}
{\texttt{vanessa.jacquier@unipd.it}}},\hspace{2pt}

Wioletta M. \textsc{Ruszel}\thanks{Utrecht University, Budapestlaan 6, 3584 CD Utrecht, The Netherlands. \newline $_{}$\hfill     \href{mailto:w.m.ruszel@uu.nl}
{\texttt{w.m.ruszel@uu.nl}} } \hspace{5pt}
}

\begin{document}

\maketitle
\begin{abstract}
We investigate the Ising model on finite subgraphs of the hyperbolic lattice under minus boundary conditions and in the presence of a positive external field $h$.
Interpreting the boundary as frozen or cold wall conditions, we show that, for small values of $h$, the system exhibits metastable behaviour. Our result is very surprising, since non-amenable graphs, such as hyperbolic lattices, feature exponentially growing boundaries, which typically destabilize local energy minima.
In particular, we identify the unique metastable state and characterize the exit time from it. Furthermore, we establish asymptotic results for the distribution of the first hitting time and provide estimates for the spectral gap. Finally, we analyze the energy landscape and describe the nucleation mechanism for values of $h$ outside the metastable regime. 
 \\

\noindent   {\it MSC2020}  {\it Subject classifications:} 82B20, 82C20, 82C26, 60K35, 51M10, 60J10, 05B45 \\
    
\end{abstract}

\section{Introduction}

\subsubsection*{Metastability phenomena and mathematical approach}
A thermodynamical system influenced by noisy dynamics is said to exhibit metastable behavior when it stays for a long period near a local minimum of the energy, before making a sudden transition to a more stable state. On short time scales, the system appears to be in equilibrium, but, on long time scales, it shifts between different regions of its state space. This evolution is typically associated with first-order phase transitions and is triggered when a critical microscopic configuration emerges, either due to spontaneous fluctuations or external perturbations.

Metastability is observed in a variety of thermodynamical systems, such as magnets under an external magnetic field, supercooled liquids, and supersaturated gases. However, it is by no means limited to thermodynamics, similar phenomena arise across many disciplines including biology, chemistry, computer science, and economics.

To characterize metastable behavior, one needs to analyze the transition time to the stable phase, identify the critical configurations that trigger this transition, and describe the typical trajectories followed by the system during such events. Two main approaches have traditionally been used to study these questions: the pathwise approach the pathwise approach (e.g., \cite{cerf2013nucleation, dehghanpour1997metropolis, kotecky1993droplet, manzo2004essential, manzo1998relaxation, manzo2001dynamical, schonmann1994slow, schonmann1998wulff, nardi1996low, olivieri2005large, jovanovski2017metastability}) and the potential-theoretic approach (e.g., \cite{bovier2016metastability, bovier2010homogeneous, gaudilliere2020asymptotic, bashiri2017note}). More recently, additional techniques have been developed, as seen in \cite{beltran2010tunneling, beltran2015martingale, beltran2012tunneling, gaudillierelandim2014, bianchi2016metastable, bianchi2020soft}. 

This method focuses on the dynamical mechanisms driving the transition from a metastable to a stable phase, and has also been extended to study transitions between different stable states. It provides a detailed analysis of the three central aspects of metastability: the typical transition time, the structure of the critical configurations, and the ensemble of trajectories followed by the system during the transition. Important developments and applications of this approach can be found in \cite{cirillo2003metastability, cirillo2015metastability, manzo2004essential, nardi2016hitting}.

The second influential method is the potential-theoretic approach, introduced in \cite{bovier2002metastability} and elaborated in the monograph \cite{bovier2016metastability}. Rather than focusing directly on transition paths, this method employs tools from potential theory to study the dynamics of the system through the analysis of hitting times of metastable sets. Central to this framework is the concept of capacity, which can be estimated via variational principles and yields particularly sharp estimates of the expected transition time, often more precise than those obtained through the pathwise approach.

These two frameworks are not equivalent, as they rest on different conceptualizations of metastable states. The distinction becomes especially relevant in more delicate contexts, such as infinite-volume systems, irreversible dynamics, or dynamics with degeneracies, as discussed in \cite{cirillo2013relaxation, cirillo2015metastability, cirillo2017sum}. 

Applications of the pathwise approach to finite-volume systems at low temperature can be found in \cite{apollonio2022metastability, bet2021critical, bet2022metastability, bet2024metastability, cirillo2024homogeneous, cassandro1984metastable, cirillo1998metastability, neves1992behavior} for single-spin-flip Glauber dynamics, and in \cite{cirillo2003metastability, cirillo2008metastability, cirillo2008competitive, bet2020effect} for parallel dynamics, such as probabilistic cellular automata. A thorough review covering both serial and parallel dynamics is provided in \cite{cirillo2022metastability}. A thorough review covering both serial and parallel dynamics is provided in \cite{cirillo2022metastability}, while a survey of metastable behavior of the Ising model and its main results available in the literature can be found in \cite{jacquier2025exploring}.

The potential-theoretic framework has likewise been successfully applied to finite-volume models at low temperature in \cite{bet2020effect, bovier2006sharp, bovier2002metastability, den2011kawasaki, den2012metastability, nardi2012sharp}. Moreover, the more intricate regime of infinite volume, with either vanishing magnetic fields or low temperatures, has been the subject of several investigations within the context of Ising-like models, particularly under Glauber and Kawasaki dynamics; see for instance \cite{baldassarri2023metastability, bovier2010homogeneous, cerf2013nucleation, gaudilliere2009ideal, gaudilliere2020asymptotic, gaudilliere2010upper, hollander2000metastability, schonmann1998wulff}.

\subsubsection*{Ising models on hyperbolic lattices}

Hyperbolic lattices $\mathcal{L}_{p,q}$, where $\frac{1}{p}+ \frac{1}{q} < \frac{1}{2}$, play a central role in geometry, topology, and mathematical physics when exploring spaces beyond the Euclidean setting, most notably, those with constant negative curvature. Concrete applications have emerged in fields such as crystallography \cite{osti_1979736}, non-Euclidean analogues of the quantum spin Hall effect \cite{cryst}, and quantum electrodynamics \cite{electro}, with striking experimental implications \cite{chen23}. An example of such a lattice can be found in Figure \ref{fig:lattice}.

These lattices correspond to discrete symmetry groups acting on the hyperbolic plane, forming regular tessellations of two-dimensional spaces with constant Ricci curvature -1, such as the Poincar\'e disc. They also represent the simplest examples of regular lattices where each face has 
$p$ sides and each vertex has degree $q$ within a non-Euclidean geometric framework. The isoperimetric constant $i_e(\mathcal{L}_{p,q})$ for these lattices is positive (contrary to e.g. $\mathbb{Z}^d$) and was computed in \cite{haggstrom2002explicit}. This means that no finite subset has a {\it small boundary} relative to its size. In fact, the boundary and the size are of the same order. 

The ferromagnetic nearest-neighbors Ising model on these lattices also behaves much differently from its counterpart on $\mathbb{Z}^2$ at low temperature.  
The authors in \cite{RNO} studied Ising models with different boundary conditions on hyperbolic graphs. They find that key thermodynamic quantities such as magnetization and susceptibility diverge at a critical temperature with classical mean-field exponents, which is attributed to the effectively infinite dimensionality of hyperbolic lattices. They give field theoretic conjectures that even at criticality, correlations decay exponentially. 
Moreover, we employ a constructive method based on layer decomposition, introduced in \cite{RNO}, in which hyperbolic lattices are generated by iteratively adjoining layers of tiles around a fundamental region. The configuration of vertices in each layer follows a recursive relation, and this recursive structure will serve as a key ingredient in our analysis.

Moreover, in \cite{wu} it is shown that the Ising model undergoes multiple phase transitions, a phenomenon absent in the classical Euclidean setting. Specifically, for sufficiently large tree degree, the system evolves from a unique Gibbs state with zero magnetization at high temperatures, to a symmetry-broken regime characterized by non-zero magnetization and nontrivial two-point correlations, and finally to a low-temperature phase in which the Gibbs state splits symmetrically into plus and minus extremal states.
The presence of an intermediate phase, where the Gibbs state with free boundary conditions differs from the average of the extremal states, highlights novel effects induced by the hyperbolic geometry. Altogether, these results uncover a much richer phase structure in negatively curved lattices compared to flat ones, with methods that also extend to Potts models under analogous assumptions. The analysis of the phase transition was later refined in \cite{monroe}.

Regarding extremality, Aizenmann~\cite{aizenman1980} and Higuchi~\cite{higuchi1979} proved that the set of \emph{extremal} Gibbs states at low temperature consists of two measures concentrated on the all plus and all minus configuration in $\mathbb{Z}^2$. For hyperbolic lattices, it is now known that there are uncountably many extremal Gibbs states at low temperature~\cite{DCLN}, indexed by certain bi-infinite geodesics on the dual lattice (similar to Dobrushin interfaces). 

In the contest of metastability, the authors of \cite{martinelli2004glauber} show that, on the regular tree with “plus” boundary conditions, both the spectral gap and the log-Sobolev constant are uniformly positive and independent of the system size. It follows that the mixing time is at most of order $O(\log(n))$, where $n$ denotes the radius of the considered ball, thus excluding the presence of metastable behavior.

Similarly, in \cite{berger2005glauber}, the Glauber dynamics of the Ising model on trees and hyperbolic graphs is analyzed, showing that the relaxation time (the reciprocal of the spectral gap) grows at most polynomially in $n$. This result confirms that, even in negatively curved geometries, mixing occurs rapidly, without the exponential times characteristic of metastability.

In the present work, however, the situation changes substantially. The imposition of unfavorable boundary conditions (equivalent, in our Hamiltonian formulation \eqref{def:Hamiltonian}, to having minus spins fixed on the outer boundary while the external field is positive) induces a competition between the bulk and the boundary: the bulk tends to align in $+1$ due to the field, while the large hyperbolic boundary exerts a pull toward $-1$. This configuration can generate a metastable well: the homogeneous negative state becomes quasi-stable, with a prolonged residence time before a cluster of plus spins reaches the critical radius necessary for nucleation. In contrast, in \cite{martinelli2004glauber, berger2005glauber}, the boundary conditions considered are either free or fixed positive, and no scenario with unfavorable boundaries (all minuses) is treated.

\subsubsection*{Our results}
In this work, we study the Ising model on the hyperbolic lattices $\mathcal{L}_{p,q}$ at very low temperature, with isotropic interactions, fixed boundary conditions and under the influence of a weak external magnetic field. The system is assumed to evolve according to Glauber dynamics. 

We consider a finite lattice $\Lambda \subset \mathcal{L}_{p,q}$ consisting of the union of $N+1$ layers. Each layer consists of vertices which can be of two different types. Vertices in $I_{k;p,q}$ in the $k$-th layer are connected to vertices in the previous $k-1$-th layer whereas vertices in $E_{k;p,q}$ are not.

The Hamiltonian, defined as
\begin{align}
    H(\sigma):=-\frac{1}{2}\sum_{\substack{i,j \in \Lambda\\ d(i, j)=1}} \sigma (i) \sigma (j) +\frac{q-3}{2}\sum_{\substack{i \in I_{N;p,q}}} \sigma (i) +\frac{q-2}{2}\sum_{\substack{i \in E_{N;p,q}}} \sigma (i) -\frac{h}{2} \sum_{i \in \Lambda} \sigma (i),
\end{align}
is chosen to promote homogeneous nucleation from the homogeneous state $\textbf{-1}$ toward the stable state $\textbf{+1}$. 

In this formulation, the Hamiltonian function imposes fixed boundary conditions: this can be interpreted as if the spins outside the domain are frozen at a constant value, in our case $-1$.

We prove that the system exhibits metastability in a small but critical range of parameters, which is a surprising and novel result.
Roughly speaking, we will prove that:
\begin{thm*} (Identification of metastable states) For $h$ in some specific critical region, the metastable state is equal to the homogeneous state
$\textbf{-1}$ and the maximal stability level $\Gamma^{p,q}$ is an explicit constant depending on $p,q$. 
\end{thm*}
\begin{proposition*}(Recurrence property). Let $K^*$ be an explicit constant depending on $p,q,h$. Then, the dynamics sped up by a time factor of order $e^{\beta K^*}$ reaches with high probability the set $\{ \textbf{-1}, \textbf{+1} \}$.
\end{proposition*}
The following theorem provides information about the asymptotic behavior (for $\beta \to \infty$) of the transition time  for the system started in the metastable state. 
\begin{thm*}(Asymptotic behavior of $\tau_{\textbf{+1}}$ in probability) 
For the system started in the configuration $\textbf{-1}$, with high probability the hitting time behaves for $\beta \rightarrow \infty$ as
\[
\tau_{\textbf{+1}} \approx e^{\beta \Gamma^{p,q}}.
\]
\end{thm*}
We also have a statement of the asymptotic distribution of $\tau_{\textbf{+1}}$, mixing time and spectral gap.

A schematic representation of the energy landscape for the critical region of $h$ is depicted in Figure \ref{fig:intro_schema}.
After creating a critical droplet which consist of a positive combinatorial ball of radius $r^*$ with a strip in the subsequent layer (whose length is a function of $h$), the nucleation process proceeds as follows. Swapping a minus into a plus close to the cluster (in the $(r^*+1)$-st layer) the energy will increase and decrease according to where the $+1$ was added. Completing the ball of radius $r^*+1$ we arrive at a local minimum in the energy landscape. Then, an additional $+1$ is added in the $(r^*+2)$-nd layer and the energy increases, and so on until the whole lattice is nucleated.

\begin{figure}
\centering
\includegraphics[scale=0.4]{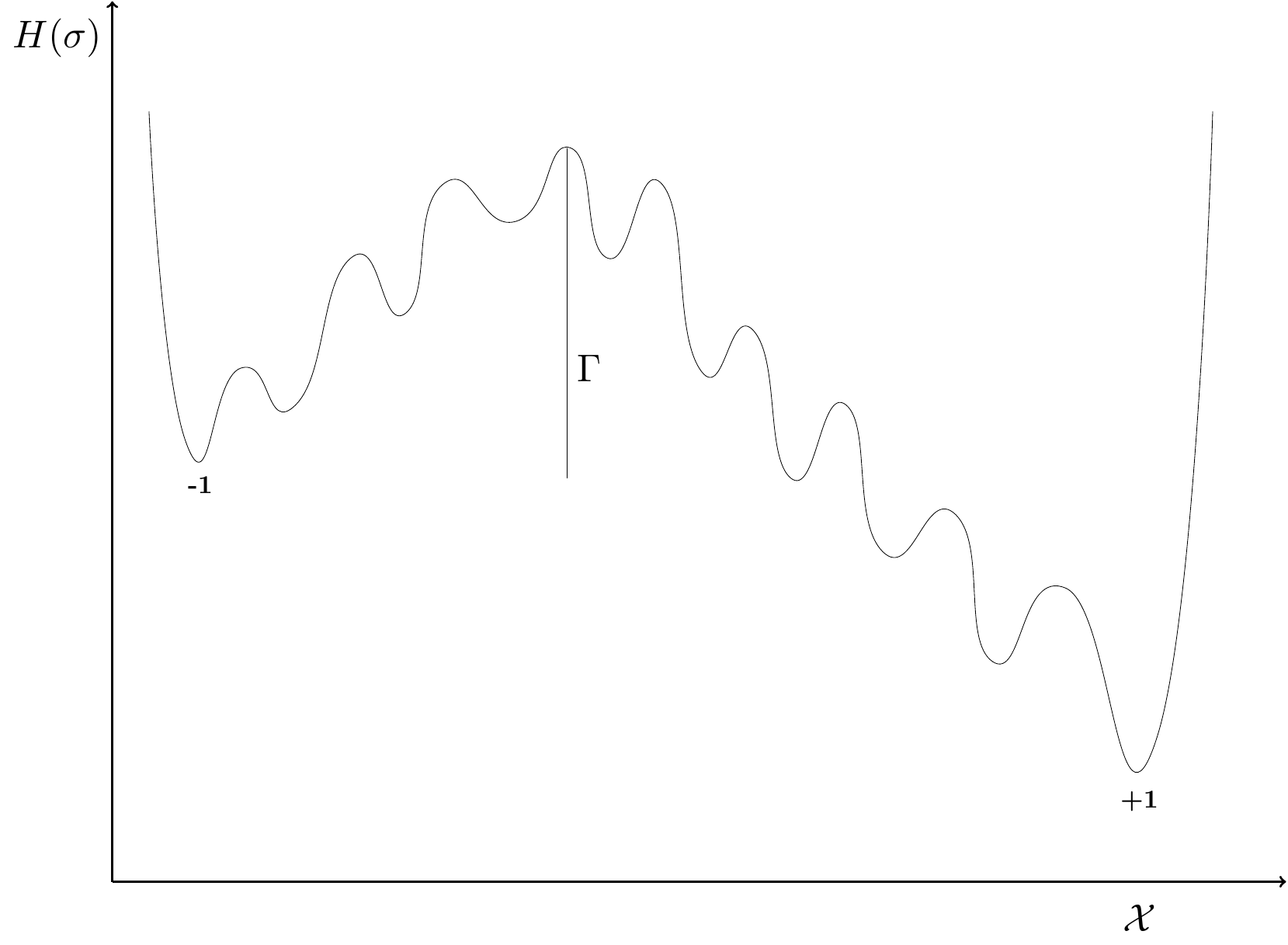}
\caption{Schematic depiction of the energy landscape in the critical region.} \label{fig:intro_schema}
\end{figure}

 In standard Euclidean lattices, the energetic cost of a droplet is dominated by two contributions: the interface energy, which grows linearly with the perimeter, and the bulk energy due to the external magnetic field, which grows with the area. Typically, the interface term dominates for small droplets, creating a barrier that stabilizes metastable states. On non-amenable graphs, such as regular trees or hyperbolic lattices, the situation is fundamentally different: the boundary of a finite region grows exponentially with its volume, so that the perimeter and the area are of the same order. As a result, the interface no longer provides a stabilizing barrier, making nucleation much easier and preventing the usual metastable traps, see \cite{martinelli2004glauber, berger2005glauber}. However, in our case, unfavorable boundary conditions (minus spins on the outer boundary with a positive external field) induce a competition between bulk and boundary, yielding a metastable well in which the homogeneous negative state persists until a cluster of plus spins reaches the critical nucleation.

These results represent a significant conceptual novelty, as it shows that the common conclusion \emph{there is no metastability on trees} is not universal, but rather depends sensitively on the choice of boundary conditions. In hyperbolic geometries, metastability is not eliminated, but its existence is shifted: the critical factor lies in the competition between the boundary influence and the external field, rather than solely in the internal energy balance.

Our findings highlight a more subtle aspect of dynamics on exponentially growing spaces. The boundary does not act merely as a perturbation favoring alignment, but can function as an active structural element that generates competition. Depending on the system parameters, this competition may either produce or suppress metastable behavior.

From a broader perspective, this observation has implications beyond the specific lattice model considered here. It suggests new mechanisms for controlling relaxation times in complex or disordered materials and provides a conceptual framework for understanding slow dynamics in hierarchical or hyperbolic networks, where the interplay between local influences and global fields can give rise to long-lived metastable states. 

In general, our work demonstrates that the relationship between metastability and hyperbolic geometry is more nuanced than previously thought and emphasizes the structural role of boundary conditions in determining the behavior of the system.

From a physical perspective, such conditions can model the interaction with materials or surfaces that maintain a stable spin polarization. A typical example arises in crystal growth or wetting phenomena, where a fixed boundary represents a substrate that either favors or inhibits the adhesion of a new phase.

When the boundaries of the system are fixed in a phase opposite to the one intended to nucleate or propagate in the bulk, they are referred to as \emph{cold walls}. This term metaphorically describes an environment that actively resists phase transitions, thereby stabilizing the initial configuration (in this case $\textbf{-1}$). The presence of such walls has a significant impact on the dynamics: it makes the potential well associated with the metastable state particularly deep. As a result, the average time required for the system to escape from this state grows exponentially, in accordance with classical metastability theory.

Moreover, we note that, on the hyperbolic graph, it is not possible to impose periodic boundary conditions without altering the degree of the vertices or modifying the structure of the tiles at the inner boundary of the lattice.

On the other hand, the use of free boundary conditions would induce preferential nucleation at the boundary. Indeed, the energy cost associated with flipping a spin at a boundary vertex is significantly lower than that of a spin in the bulk. This is because a boundary vertex contributes a much smaller increase in perimeter, being connected only to the interior of the lattice and not interacting with external spins, as dictated by free boundary conditions.

Recalling that the hyperbolic lattice is defined for $1/p+1/q<1/2$, we have excluded from our study the case $p=3,q \geq 7$. We believe this choice is purely technical, and that results similar to those analyzed here would hold. More precisely, a unique representation of the hyperbolic lattice is provided in \cite{RNO} for all $p,q \geq 4$, while for $p=3$ the representation is slightly different. Indeed, in our case, each layer is partitioned into two sets as well: vertices not connected to the previous layer, and vertices connected to a single vertex in the previous layer. For $p=3$, each layer can be partitioned into two sets of vertices: those connected to only one vertex in the previous layer, and those connected to two vertices in the previous layer. This slightly alters the geometry and the associated parameters, but we expect that the results presented for our case can be extended in a similar manner to the case $p=3, q \geq 7$.

\subsubsection*{Organization of the paper}
The paper is organized as follows. In Section~\ref{sec:model}, we introduce the model and we analyze the energy landscape. We present the main results in Section~\ref{main}. Section~\ref{geom1} is dedicated to the description of the model geometry. In particular, we create a bijection between the clusters in the Ising configuration and the polyamonds. 
In Section \ref{proofs}, we prove the main results in the metastable parameter region. More precisely, Section~\ref{sec:recurrence} is devoted to proving theorems regarding the recurrence of the system to either the stable or metastable state. Section~\ref{bound} focuses on identifying the reference path and estimating the energy barrier.
The remaining theorems concerning the metastable behavior of the system are proven in Section~\ref{sec:other_theorems}.
In Section \ref{sec:other_regions}, we describe the evolution of the system in the remaining parameter regimes of $h$. Finally in the appendix, Section \ref{sec:appendix}, we show in explicit examples how two different critical droplets arise for different magnetic fields $h$ and the same $p,q$.

\section{Model description and analysis of the energy landscape}\label{sec:model}
\subsection{Ising model on $\mathcal{L}_{p,q}$}
Consider the hyperbolic plane $\mathbb{H}^2$ with origin $\textbf{o}$ (i.e.~the unique 2-dimensional Riemannian manifold of constant negative sectional curvature, which we fix here to be $-1$).
Denote by $\mathcal{L}_{p,q}$, with $1/p+1/q < 1/2$, the tessellation of $\mathbb{H}^2$ made of regular $p$-gons whose sides have unitary hyperbolic length and in which each vertex has degree $q$, see Figure \ref{fig:lattice} for instance. 
\begin{figure}[!hbtp]
    \centering
    \begin{minipage}{0.45\textwidth}
        \centering
        \includegraphics[width=0.74\textwidth]{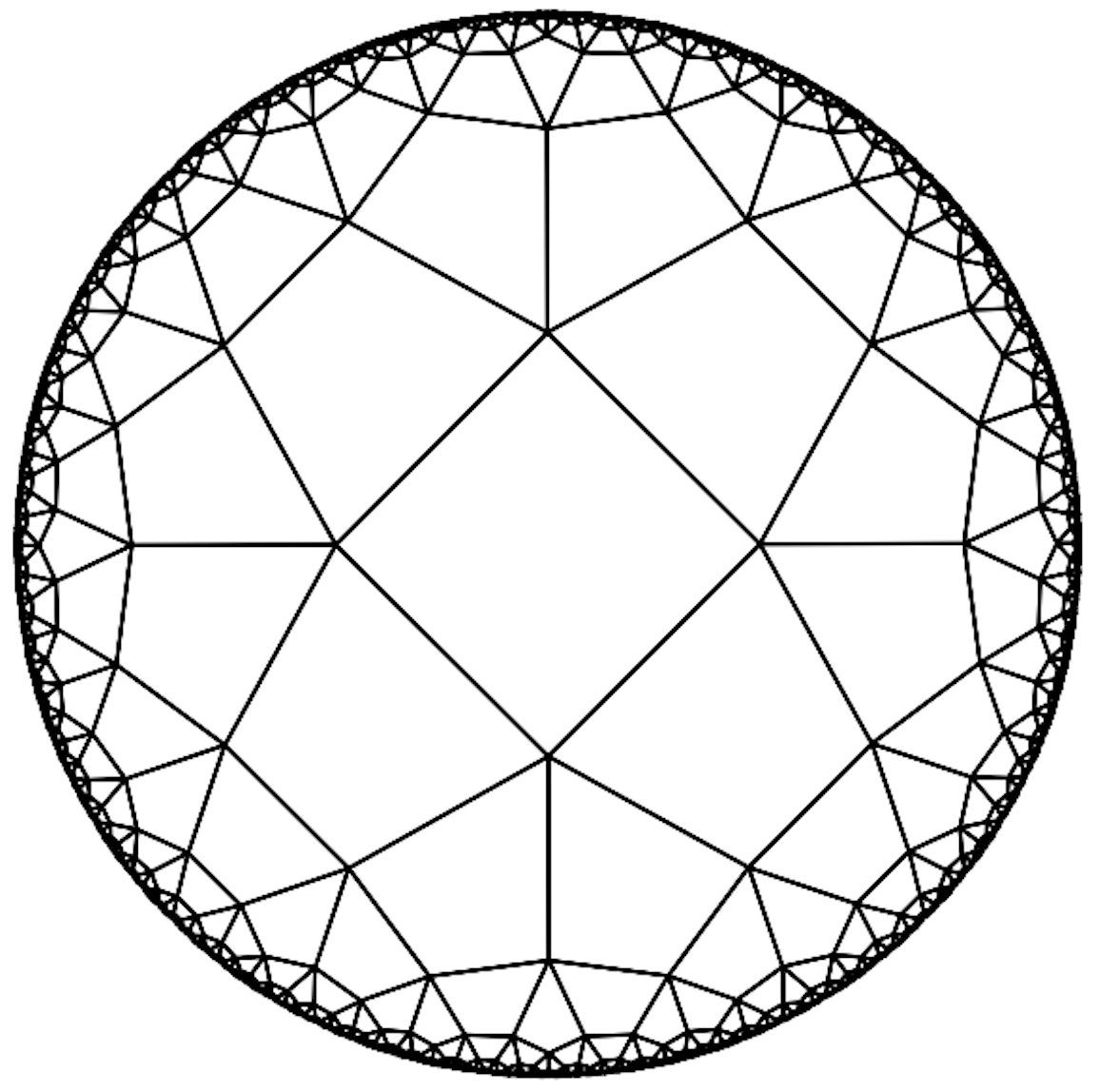} 
        \caption{Embedding of $\mathcal{L}_{4,5}$ in the hyperbolic disc}\label{fig:lattice}
    \end{minipage}\hfill
    \begin{minipage}{0.45\textwidth}
        \centering
        \includegraphics[width=0.74\textwidth]{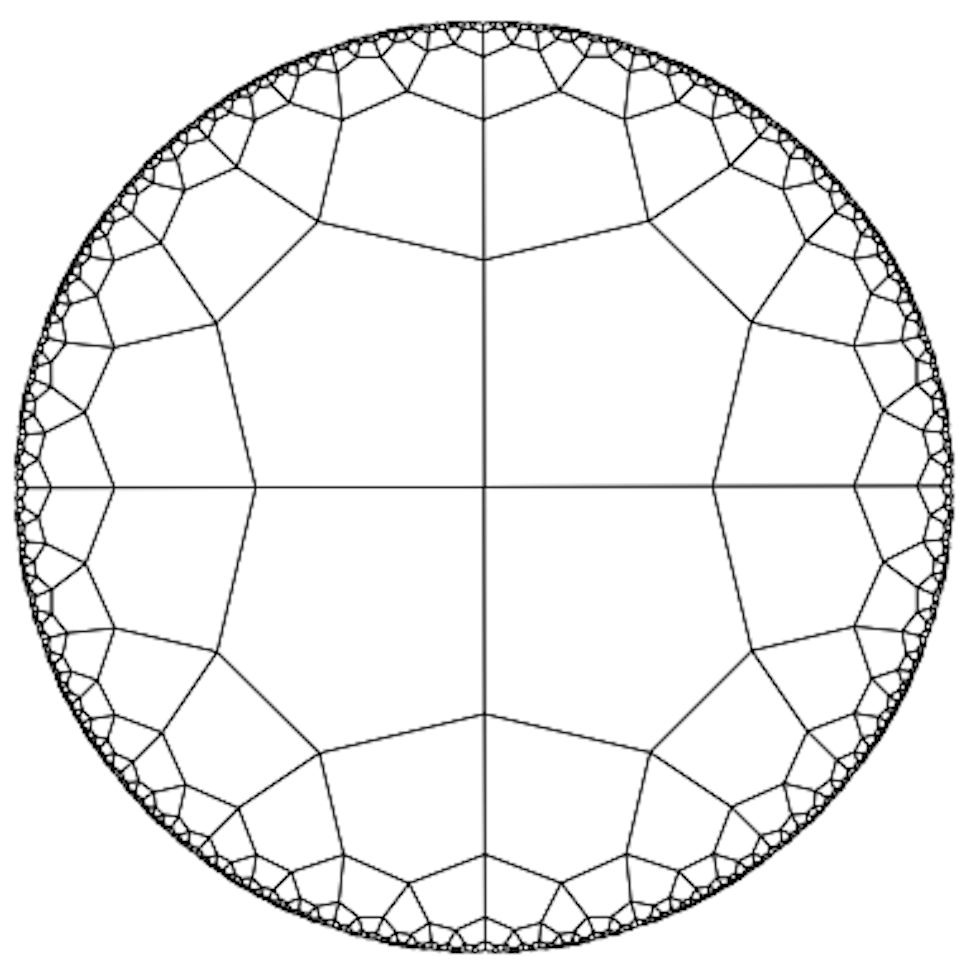} 
        \caption{$\mathcal{L}_{5,4}$ as the dual lattice of $\mathcal{L}_{4,5}$.} \label{fig:dual}
    \end{minipage}
\end{figure}
Assume that $p, q \geq 4$ and write $\mathcal{L}_{p,q}=(\mathcal{V}_{p,q}, \mathcal{E}_{p,q})$ with vertex set  $\mathcal{V}_{p,q}$ and edge set $\mathcal{E}_{p,q}$. 
We note that the dual lattice of $\mathcal{L}_{p,q}$ contains the origin and is equal to the graph $\mathcal{L}_{q,p}$, see Figure \ref{fig:dual}.

Let $k\in \mathbb{N} \cup \{0\}$ and let $v$ be a vertex in the dual lattice $\mathcal{L}_{q,p}$.  We define the $k$-th \emph{layer} of $v$, denoted by $L_k(v)$, as the set of vertices in $\mathcal{V}_{p,q}$ constructed in the following way. The zero-th layer $L_0(v)$ is the set of $p$ vertices in the unique face of $\mathcal{L}_{p,q}$ containing $v$. 
The first layer $L_1(v)$ is the set of vertices not in $L_0(v)$ which are in all the faces adjacent to the face containing $v$ (including those sharing just a vertex). In other words, $L_1(v)$ contains all vertices that are connected with the vertices in $L_0(v)$. 
For $k \geq 2$, we define $L_k(v)$ iteratively as the set of vertices which are in all the faces adjacent to the faces containing the vertices in $L_{k-1}(v)$.
Clearly, $\mathcal{V}_{p,q}=\cup_{k\geq 0} L_k(v)$ and $L_j(v) \cap L_k(v) = \emptyset$ if $j\neq k$. We denote by $B_{n;p,q}(v)=\bigcup_{i=0}^{n-1}L_i(v)$ the ball centered at $v$ of the union of the first $n$ layers. If $v$ is the origin, then we will drop the dependence on $v$.

Furthermore, set $v\equiv\textbf{o}$ and for each $n \geq 1$ and $p \geq 4$, we will use the construction in \cite{RNO} and define the following partition of $\mathcal{V}_{p,q}$:
\begin{itemize}
    \item $I_{n;p,q} = \{v \in L_n \, | \, \, \exists \, w \in  L_{n-1}: (v,w)\in \mathcal{E}_{p,q} \}$;
    \item $E_{n;p,q} = L_n\setminus I_{n;p,q}$,
\end{itemize}
so that $|L_n| = |I_{n;p,q}|+|E_{n;p,q}|$. 
We have the following recursion relation for $I_{n;p,q}, E_{n;p,q}$, see also \cite[Equations (2.14) and (2.16)]{RNO}, for $p\geq 4$ and $n\geq 0$
\begin{align}\label{eq:recur}
{I_{n+1;p,q} \choose E_{n+1;p,q}} = T {I_{n;p,q} \choose E_{n;p,q}},
\end{align}
where
\begin{align}\label{eq:T1}
T = \begin{pmatrix}
q-3 & q-2 \\
8 -3p-3q+pq & 5-2p-3q+pq
\end{pmatrix}
\end{align}
and $(I_{0;p,q},E_{0;p,q})=(0,p)$.
The eigenvalues of $T$ are given by
\begin{align}\label{eq:eigen}
\lambda_{\pm} &= \frac{1}{2} \left (2+p(q-2)-2q \pm \sqrt{(p-2)(q-2)(q(p-2)-2p)}\right ) \notag \\
&= \left ( \frac{1}{2}(p-2)(q-2)-1\right ) \left ( 1 \pm \sqrt{1-\frac{4}{((p-2)(q-2)-2)^2} }\right ).
\end{align}

Therefore,
\begin{align}
&|I_{n;p,q} |= 2c_{p,q} \sqrt{q-2}( \lambda^{n}_+ - \lambda^{n}_-), \\
&|L_n|= (1,1)T^n {0\choose p} =c_{p,q} ( a_-\lambda^n_{-} + a_+ \lambda^n_+),
\end{align}
where
\begin{align}\label{constant_cpq}
    c_{p,q}=\frac{p}{2 \sqrt{(p-2)(pq-2p-2q)}},
\end{align}
and
\begin{align}\label{eq:apm}
a_{\pm} &=(\sqrt{(p-2)(p(q-2)-2q)} \mp 2 \sqrt{q-2} \pm p\sqrt{q-2}) \notag \\
& =(p-2)\sqrt{q-2} \left (\pm1+\sqrt{1-\frac{4}{(p-2)(q-2)}}  \right).
\end{align}

Let $N\in \mathbb{N}$ and let $\Lambda \subset \mathcal{L}_{p,q}$ be defined as $\Lambda := B_{N+1;p,q}(\textbf{o})$. Denote by $\sigma:\Lambda \longrightarrow \{-1,+1\}^{\Lambda}$ an \textit{Ising spin configuration} on $\Lambda$ and let $\mathcal{X}:=\{-1,+1\}^{\Lambda}$ be the {\it configuration space}.

\begin{definition}
Let $p,q\geq 4$. The Hamiltonian function $H: \mathcal{X} \longrightarrow \mathbb{R}$ for the Ising model on $\Lambda$ is defined as
\begin{equation}\label{def:Hamiltonian}
H(\sigma):=-\frac{1}{2}\sum_{\substack{i,j \in \Lambda\\ d(i, j)=1}} \sigma (i) \sigma (j) -\frac{h}{2} \sum_{i \in \Lambda} \sigma (i) +\frac{q-3}{2}\sum_{\substack{i \in I_{N;p,q}}} \sigma (i) +\frac{q-2}{2}\sum_{\substack{i \in E_{N;p,q}}} \sigma (i) ,
\end{equation} 
where $h\in \mathbb{R}$ is a positive parameter and $d(\cdot, \cdot)$ is the graph distance on $\mathcal{L}_{p,q}$. 
\end{definition}

The first two sums are the same as in the definition of a {\it classical} Hamiltonian with $J=1$. The remaining two sums are over vertices in the boundary layer (which has the same cardinality order as the bulk) and are chosen such that the Hamiltonian is homogeneous, i.e. $H(\sigma) = H(\sigma')$ for $\sigma, \sigma'$ which differ by swapping a minus spin uniformly in $\Lambda$.

Moreover, we define the Gibbs measure for our model as follows.

\begin{definition}\label{def:gibbs}
  The Gibbs measure of the Ising model with Hamiltonian $H$ on $\Lambda$ with free boundary conditions is defined by 
\begin{equation}\label{eq:gibbs}
    \mu(\sigma)=\frac{e^{-\beta H(\sigma)}}{Z_{\beta}},
\end{equation}
where $\beta:=\frac{1}{T} >0$ is the inverse temperature and $Z_{\beta} = \sum_{\eta \in \mathcal{X}} e^{-\beta H(\eta)}$ is the normalization constant. 
\end{definition}

\subsection{Glauber dynamics}
We study the evolution of our model under Glauber dynamics. In particular, we consider a Markov chain  $(X_t)_{t \in \mathbb{N}}$ on $\mathcal{X}$ 
defined via the so called \emph{Metropolis Algorithm}.
The transition probabilities of this dynamics are given by
\begin{equation}\label{transitionprob}
    p(\sigma, \eta)=f(\sigma,\eta) e^{-\beta[H(\eta) -H(\sigma)]_+}, \qquad \text{for all } \sigma \neq \eta,
\end{equation}
where $[\cdot]_+$ denotes the positive part and $f(\sigma,\eta)$ is a connectivity matrix
independent of $\beta$, defined, for all $\sigma \neq \eta$, as
\begin{equation}
    f(\sigma,\eta)= \left\{
    \begin{array}{ll}
    \frac{1}{|\Lambda|} & \;\;\textrm{ if } \exists \; i\in \Lambda: \sigma^{(i)}=\eta\\
    0& \;\;\textrm{ otherwise }
    \end{array}
    \right.
\end{equation}
where
\begin{equation}
    \sigma^{(i)}(j)= \begin{cases}
    \begin{array}{ll}
        \sigma(j) & \;\;\textrm{ if } j \neq i\\
        -\sigma(j) & \;\;\textrm{ if } j = i.
        \end{array}
    \end{cases}
\end{equation}
We will say that two configurations $\sigma, \eta$ are \emph{communicating} when they differ by exactly one spin, i.e. $\eta= \sigma^{(i)}$ for some $i \in \Lambda$.

In the next lemma, we prove that  $(X_t)_{t \in \mathbb{N}}$ is reversible with respect to the Gibbs measure defined in Definition \ref{def:gibbs}.

\begin{lemma}\label{lem:reversibility}
$(X_t)_{t \in \mathbb{N}}$ is an ergodic and aperiodic Markov chain  on $\mathcal{X}$ satisfying the detailed balance condition
\begin{equation}\label{reversibility}
    \mu(\sigma)p(\sigma, \eta)=\mu(\eta)p(\eta, \sigma),
\end{equation}
where $\mu (\cdot)$ is the Gibbs measure \eqref{eq:gibbs}.
\end{lemma}
\begin{proof}[Proof of Lemma \ref{lem:reversibility}]
    Let $\sigma$ and $\eta$ be two communicating configurations. If $\sigma=\eta$ the statement is trivial. We consider the case $\sigma\neq \eta$, thus $\eta=\sigma^{(i)}$ for some $i\in \Lambda$. First we discuss the case $i\in L_k$ for $k=0,\ldots,N-1$ and then the case $i\in L_N$. In the first case we have
\begin{equation}\label{eq:rev1}
\mu_\beta(\sigma)p_\beta(\sigma,\sigma^{(i)})
=
\frac{e^{-\beta H(\sigma)}}{Z_\beta}
\frac{e^{-\beta[H(\sigma^{(i)})-H(\sigma)]_+}}{|\Lambda|}
=
\frac{e^{-\beta H(\sigma)}}{Z_\beta}
\frac{e^{-\beta[\sigma(i) (h-q+2k)]_+}}{|\Lambda|}
\end{equation}
where $k\in \{0,...,q\}$ is the number of the nearest sites with the same spin value. Moreover, we have that
\begin{equation}\label{eq:rev2}
\mu_\beta(\sigma^{(i)}) p_\beta(\sigma^{(i)},\sigma)
=
\frac{e^{-\beta H(\sigma^{(i)})}}{Z_\beta}
\frac{e^{-\beta [-\sigma(i) (h-q+2k)]_+}}{|\Lambda|}.
\end{equation}
By the definition of the Hamiltonian in Equation \eqref{def:Hamiltonian}, we obtain $H(\sigma^{(i)})=H(\sigma)+\sigma(i) (h-q+2k)$ and we get that Equations \eqref{eq:rev1} and \eqref{eq:rev2} are equal.

For the second case, when $i\in L_N$ we have that
\begin{equation}\label{eq:rev3}
\mu_\beta(\sigma)p_\beta(\sigma,\sigma^{(i)})
=
\frac{e^{-\beta H(\sigma)}}{Z_\beta}
\frac{e^{-\beta[\sigma(i) (h-(q-a)+b)]_+}}{|\Lambda|}
\end{equation}
where 
\begin{align}
    a= 
    \begin{cases}
        2 \qquad \text{ if } i\in E_{N;p,q} \\
        3 \qquad \text{ if } i\in I_{N;p,q}
    \end{cases}
\end{align}
and $b\in \{-a,-(a-2),a-2,+a\}$ is the difference between the number of nearest neighbour sites with the same spin and the number of nearest neighbour sites with opposite spin. Furthermore,
\begin{equation}\label{eq:rev4}
\mu_\beta(\sigma^{(i)}) p_\beta(\sigma^{(i)},\sigma)
=
\frac{e^{-\beta H(\sigma^{(i)})}}{Z_\beta}
\frac{e^{-\beta [-\sigma(i) (h-(q-a)+b)]_+}}{|\Lambda|}.
\end{equation}
Again using the Definition  \ref{def:Hamiltonian}, we obtain $H(\sigma^{(i)})=H(\sigma)+\sigma(i) (h-(q-a)+b)$ and we get that the Equations \eqref{eq:rev3} and \eqref{eq:rev4} are equal.
\end{proof}

\subsection{Metastability problem}
When studying the metastable behavior of a model, we are primarily interested in the exit time from metastable states to the stable state. To this end, we define the \emph{first hitting time} of a set
$A\subset \mathcal{X}$ starting from $\sigma \in \mathcal{X}$
\begin{equation}\label{fht}
    \tau^{\sigma}_A:=\inf\{t>0 \,|\, X_t\in A\}.
\end{equation}
Whenever possible we will omit the superscript denoting the starting point $\sigma$ from the notation, and denote by $\mathbb{P}_{\sigma}(\cdot)$ and $\mathbb{E}_{\sigma}[\cdot]$ the probability and the expectation along the trajectories of the process starting at $\sigma$, respectively. 

We focus on the first arrival time of the Markov chain 
$(X_t)_{t \in \mathbb{N}}$
to the set of the {\it stable states}, which correspond to the set of global minima of $H$, starting from an initial local minimum. 

Local minima can be ordered according to their increasing stability level, which is determined by the height of the energy barrier separating them from states with lower energy. More precisely, for any $\sigma \in \mathcal{X}$, let $\mathcal{I}_{\sigma}$ be the set of configurations with energy strictly lower than $H(\sigma)$, i.e.,
\begin{equation}\label{I}
\mathcal{I}_{\sigma}:=\{\eta\in \mathcal{X} \,|\, H(\eta)<H(\sigma)\}.
\end{equation}
Let $\underline{\omega}=\{\omega_1,\ldots,\omega_n\}$ be a finite sequence of configurations in $\mathcal{X}$, where, for each $k=1,...,n-1$, the configuration $\omega_{k+1}$ is obtained from $\omega_k$ by a single spin flip. We call $\underline{\omega}$ a \emph{path} with starting and final configurations $\omega_1$ and $\omega_n$ respectively. We denote the set of all these paths as $\Theta(\omega_1,\omega_n)$ and we indicate its length as $|\underline{\omega}|=n$. Let $A,B$ be two subsets of configurations, we denote by $\Theta(A,B)$ the set of paths between all configurations from $A$ to $B$.

To characterize the metastable states, we need to define the maximal energy along the paths. In particular,
the \emph{communication height} between two configurations $\sigma$ and $\eta$ is the minimum among the maximal energies along the paths in $\Theta(\sigma,\eta)$, i.e.,
\begin{equation}\label{minmax}
\Phi(\sigma,\eta):=\min_{\underline{\omega}\in\Theta(\sigma,\eta)}\max_{\zeta \in \underline{\omega}} H(\zeta).
\end{equation}
Similarly, the communication height between two sets $A, B \subset \mathcal{X}$ is defined as
\begin{equation}
\Phi(A,B):=\min_{\sigma \in A,\eta \in B} \Phi(\sigma,\eta).
\end{equation}
The paths in $\Theta(A, B)$ which realize the min-max problem \eqref{minmax} are called \emph{optimal paths} $ (A \to B)_{opt}$.

In order to find the local minima, we will characterize the configurations according to their \emph{stability level}, defined as
\begin{equation}
V_{\sigma}:=\Phi(\sigma,\mathcal{I}_{\sigma})-H(\sigma),
\end{equation}
If $\mathcal{I}_{\sigma}$ is empty, then we set $V_{\sigma}=\infty$. Since configurations can be classified according to their stability level, a key role is played by the set of all configurations whose stability level exceeds $K$, that is 
\begin{equation}\label{Xv}
\mathcal{X}_K:=\{\sigma \in \mathcal{X} \,\, | \,\, V_{\sigma}> K \}.
\end{equation}

Furthermore, we denote by $\Gamma_{max}$ the \emph{maximal stability level}, i.e.
\begin{equation}\label{Gamma}
    \Gamma_{max}:=\max_{\sigma\in \mathcal{X}\setminus \mathcal{X}^s}V_{\sigma},
\end{equation}
where $\mathcal{X}^s$ is the set of the \emph{stable states}, which are the global minima of the energy.

Now, we are able to introduce the \emph{metastable states}.
The metastable states are those states that attain the maximal stability level $\Gamma_{max}< \infty$, that is 
\begin{equation}\label{Xm}
    \mathcal{X}^m:=\{\sigma\in \mathcal{X}| \, V_{\sigma}=\Gamma_{max} \}.
\end{equation}
Now we define formally the \emph{energy barrier} $\Gamma$ as
\begin{equation}
    \Gamma:=\Phi(\sigma, \eta)-H(\sigma) \qquad \text{with } \sigma \in \mathcal{X}^m, \, \eta \in \mathcal{X}^s.
\end{equation}

In \cite{cirillo2013relaxation}, the authors prove the equivalence between the maximal stability level and the energy barrier.
For this reason, in the rest of the paper, we will use the notation $\Gamma$ to refer indistinctly to the two definitions.
\subsection{Metastable region}

The  {\it critical region} of $h$ for which we will prove metastability is given by $(h^*_1, h^*_2)$ where
\begin{align}\label{eq:specifich}
    &h^*_1=\frac{(q-2)|L_N|-|I_{N}|}{\sum_{j=0}^N|L_j|}, \qquad h_2^*=q-2-\frac{4 \sqrt{q-2} (\lambda_+-\lambda_-)}{a_+ \lambda_++a_- \lambda_-}.
\end{align}
It is easy to see that $(h^*_1,h^*_2)$ is not empty, indeed $h^*_2=h^*_1(1)>h^*_1(N)$ for each value of $N>1$ since $h^*_1(N)$ is a decreasing function of $N$.

Let $\mathbf{\pm 1}$ be the configuration $\sigma$ such that $\sigma(i)=\pm 1$ for every $i\in\Lambda$. Let $\mathbf{\pm 1}^{(\mp)}$ be the configuration $\sigma$ such that $\sigma(i)=\pm 1$ for every $i\in\Lambda\setminus L_N$ and $\sigma(i)=\mp 1$ for every $i\in L_N$.

We define the difference between the energy of a configuration $\sigma$ and the homogeneous state $\textbf{-1}$ to simplify the computation in the rest of the paper,
\begin{align}
    \Delta H(\sigma)=H(\sigma)-H(\textbf{-1}).
\end{align}
It easily follows that for $h \in (h^*_1, h^*_2)$
 \begin{align}
\min_{\sigma \in \mathcal{X}} \Delta H(\sigma)=\Delta H(\textbf{+1}).
    \end{align}
See also Section \ref{sec:other_regions} for more details on the other regions.

In $(h_1^*,h_2^*)$, we identify the metastable and stable states and compute the exit time from the metastable state by estimating the energy barrier that the system must overcome during the transition. This barrier is attained at the critical configurations, which can be intuitively understood as configurations containing clusters of plus spins with a ball-like shape and a peculiar radius, called \emph{critical radius}. Recall Figure \ref{fig:intro_schema} for a schematic representation of the energy landscape.

\begin{definition}\label{critical_radius}
The critical radius $r^*$ is defined as the integer value
\begin{align}
    r^*= \left \lfloor \frac{\log\left ( \frac{4 \sqrt{q-2}+a_- (q-2-h)}{4 \sqrt{q-2}-a_+(q-2-h)} \right )}{\log \left (\frac{\lambda_+}{\lambda_-} \right )} \right \rfloor ,
\end{align}
where $c_{p,q}$ is defined in Equation \eqref{constant_cpq}, $a_{\pm}$ in  \eqref{eq:apm} and $\lambda_{\pm}$ in  \eqref{eq:eigen}. The critical droplet is a ball $\mathcal{B}_{r^*;p,q}(o)$ of radius $r^*$ with a strip $S_{crit} (h)$ attached in the layer $L_{r^*}$.
The critical area is equal to $A^*:= \left|\bigcup_{j=0}^{r^*-1}L_j \right|+k(h)$, where $k(h) \in \left [1, |L_{r^*}| \right)$ is equal to $|S_{crit}(h)|$ and depends on the external magnetic field $h$.
\end{definition}
The critical droplet will be formally defined in Equation \eqref{eq:crit_drop} in the proof of Proposition \ref{teoRP}, an example of a critical droplet can be found in Figure \ref{fig:crit_drop}.
 
\begin{figure}[htb!]
\centering
\includegraphics[scale=0.23]{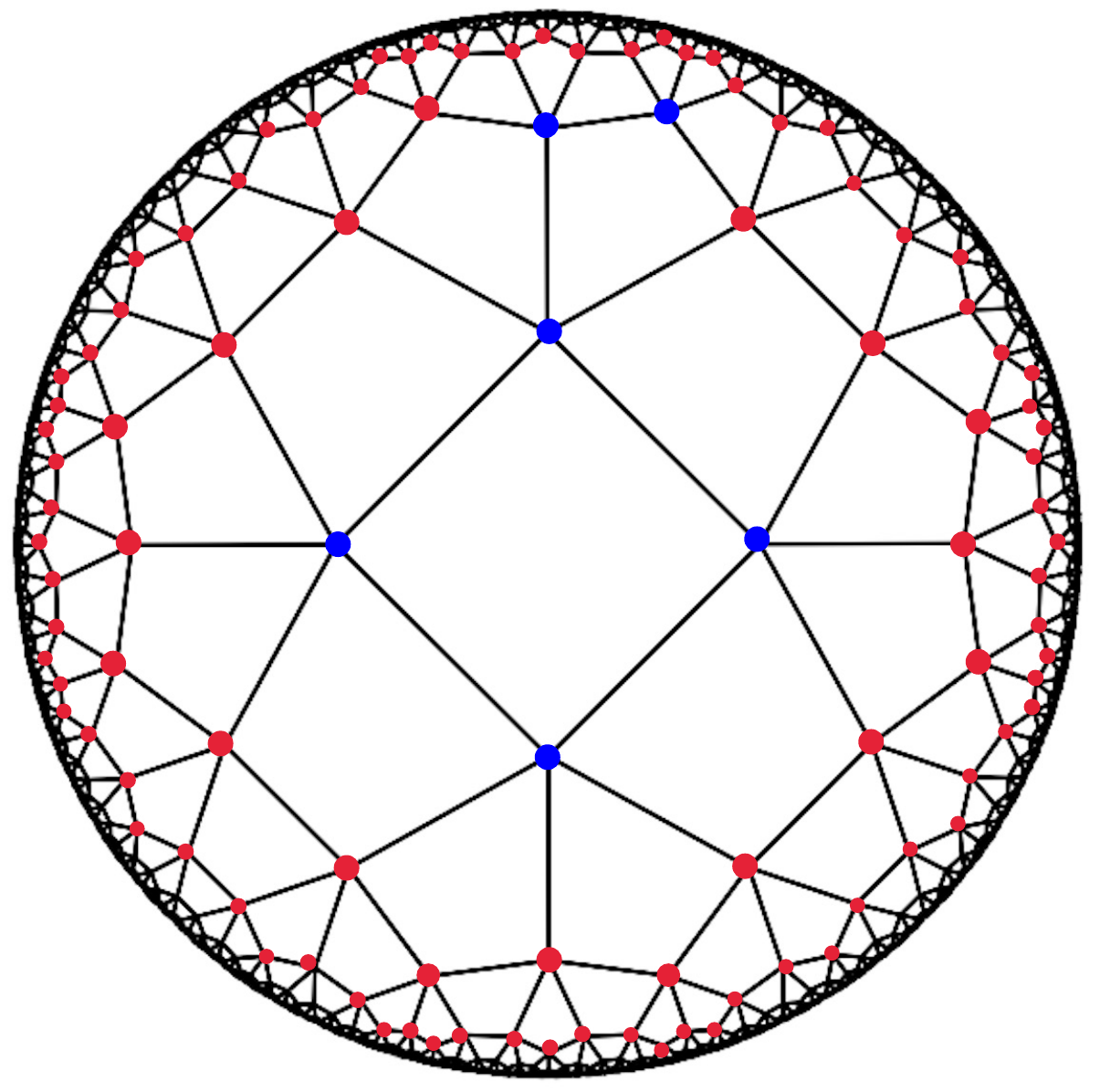}
\caption{Example of a critical droplet with radius $r^*=1$ and strip $S_{crit}(h)$ in the layer $L_1$ of length $k(h)=2$.}\label{fig:crit_drop}
\end{figure}

We note that $r^*$ is a positive integer for $h \in \left (h^*_1, h^*_2 \right )$, see Lemma \ref{remark_tildeh} for more details. The critical droplet will be explicitly constructed in the proof of the recurrence property, see Section \ref{sec:recurrence}. A specific example of a critical droplet with different strips can be found in the Appendix \ref{sec:appendix}.

\begin{lemma}\label{remark_tildeh}
The critical radius $r^*$ is a positive integer.
\end{lemma}

\begin{proof}
In order for the critical radius to be well defined we need that  
\begin{align}
        \begin{cases}
            4 \sqrt{q-2}+a_- (q-2-h)>0, \\
            4 \sqrt{q-2}-a_+ (q-2-h)>0,
        \end{cases}
    \end{align}
    and 
    \begin{align}
        \frac{4 \sqrt{q-2}+a_- (q-2-h)}{4 \sqrt{q-2}-a_+(q-2-h)}>1.
    \end{align}
    The first and the third conditions hold for each value of $h<q-2$. The second condition is equal to required
    \begin{align}
        h>q-2-\frac{4 \sqrt{q-2}}{a_+}.
    \end{align}
   
First, we observe that $h^*_1$ is a decreasing function of $N$ for each value of $p,q \geq 4$ such that $1/p+1/q<1/2$. Then, by a direct computation, we conclude
\begin{align}
    \lim_{N \to \infty} h^*_1(N)=q-2-\frac{4 \sqrt{q-2}}{a_+}.
\end{align}

It easy to see that the denominator is positive since $\lambda_+>\lambda_->0$. Then, we are left to prove that $r^* >0$, i.e.
\begin{align}
\frac{\log\left ( \frac{4 \sqrt{q-2}+a_- (q-2-h)}{4 \sqrt{q-2}-a_+(q-2-h)} \right )}{\log \left (\frac{\lambda_+}{\lambda_-} \right )}  >1
\end{align}
By a direct computation, we find that the previous inequality holds for
\begin{align}
    h<q-2-\frac{4 \sqrt{q-2} (\lambda_+-\lambda_-)}{a_+ \lambda_++a_- \lambda_-}=h^*_2.
\end{align}
\end{proof}

\section{Main results}\label{main}
In this section we present the results regarding metastability phenomena of our model in the parameter region $h \in \left (h^*_1,h^*_2 \right)$, where $h^*_1,h^*_2$ are defined in Equation \eqref{eq:specifich}.

\begin{thm}\label{Identification} (Identification of metastable states)
$\mathcal{X}^m=\{\textbf{-1}\}$ and $\Gamma_{max}=\Gamma^{p,q}$, where $\Gamma^{p,q}$ is a constant depending on $p,q$ and given by
\begin{equation}\label{def:gamma}
    \Gamma^{p,q}:=c_{p,q} \left [(a_- \lambda_-^{r^*} + a_+ \lambda_+^{r^*})(q-2-h)-4 \sqrt{q-2} (\lambda_+^{r^*}-\lambda_-^{r^*}) \right ]
    +K^*,
\end{equation}
where $a_{\pm}$ is defined in Equation \eqref{eq:apm}, $\lambda_{\pm}$ in \eqref{eq:eigen}, $c_{p,q}$ in \eqref{constant_cpq}, and $K^*$ is a constant depending on the three parameters $h,p,q$.
\end{thm}
A necessary result for validating the previous theorem is found in the following proposition, in which we state that the dynamics sped up by a time factor of order $e^{\beta K^*}$ reaches with high probability the set $\{ \textbf{-1}, \textbf{+1} \}$.
\begin{proposition}\label{teoRP} (Recurrence property). Let $K^*$ be the constant given in Equation \eqref{def:gamma}, then \mbox{$\mathcal{X}_{K^*}=\{ \textbf{-1}, \textbf{+1} \}$} and for any $\epsilon>0$
\[
\lim_{\beta \rightarrow \infty} \frac{1}{\beta}\log \left ( \mathbb{P}_{\sigma}\left(\tau_{\mathcal{X}_{V^*}}> e^{\beta(K^*+\epsilon)} \right)\right ) = -\infty.
\]
\end{proposition}
The recurrence property implies that the system reaches with high probability
either the state $\textbf{-1}$ or the  ground state $\textbf{+1}$ in a time shorter than 
$e^{\beta (K^*+\epsilon)}$, uniformly in the starting configuration $\sigma$ for any $\epsilon >0$.

In the next theorems, we give the asymptotic behavior (for $\beta \to \infty$) of the transition time for the system started in the metastable state. In particular, in Theorem \ref{thm:transitiontime} we estimate the transition time and in Theorem \ref{teotime'} we give its asymptotic distribution. 
\begin{thm}\label{thm:transitiontime} (Asymptotic behavior of $\tau_{\textbf{+1}}$ in probability) For any $\epsilon>0$, we have
\begin{equation}
    \lim_{\beta \to \infty} \mathbb{P}_{\textbf{-1}}\left(e^{\beta(\Gamma^{p,q}-\epsilon)}< \tau_{\textbf{+1}}<e^{\beta(\Gamma^{p,q}+\epsilon)} \right)=1.
\end{equation}
\end{thm}
\begin{thm}\label{teotime'} (Asymptotic distribution of $\tau_{\textbf{+1}}$) 
Let $T_{\beta}:= \inf\{ n \geq 1 \, | \, \mathbb{P}_{\textbf{-1}}(\tau_{\textbf{+1}} \leq n) \geq 1-e^{-1}\}$ 
\begin{equation}\label{Ptime'}
    \lim_{\beta \to \infty} \mathbb{P}_{\textbf{-1}}\left(\tau_{\textbf{+1}}>tT_{\beta} \right)=e^{-t}
\end{equation}
and 
\begin{equation}\label{Etime'}
    \lim_{\beta \to \infty} \frac{\mathbb{E}_{\textbf{-1}}(\tau_{\textbf{+1}})}{T_{\beta}}=1.
\end{equation}
\end{thm}
The following theorem gives an estimate of the mixing time and the spectral gap for our model.
\begin{thm}\label{TMIX} (Mixing time and spectral gap) For any $0<\epsilon<1$ we have
\begin{equation}\label{lim2PCA}
 \lim_{\beta \rightarrow \infty}{\frac{1}{\beta}\log{ t^{mix}_\beta(\epsilon)}}=\Gamma^{p,q},
\end{equation}
 and there exist two constants $0<c_1<c_2<\infty$ independent of $\beta$ such that for every $\beta>0$
\begin{equation}\label{rocompresopca}
 c_1e^{-\beta(\Gamma^{p,q}+\gamma_1)} \leq \rho_{\beta} \leq c_2e^{-\beta(\Gamma^{p,q}-\gamma_2)},
\end{equation}
 where $\gamma_1,\gamma_2$ are functions of $\beta$ that vanish for $\beta\to\infty$, and $\rho_{\beta}$ is the spectral gap.
\end{thm}

\section{Geometry of the model: clusters and polyamonds}\label{geom1}
To study the transition between $\mathcal{X}^m$ and $\mathcal{X}^s$, it is convenient
to associate to each configuration
$\sigma \in \mathcal{X}$
certain geometrical objects and analyze their properties.
To this end, recall that $\mathcal{L}_{p,q}$ is the discrete lattice embedded in $\mathbb{H}^2$ and $\mathcal{L}_{q,p}$ is its dual, also embedded in $\mathbb{H}^2$.

\begin{definition}
Given a configuration $\sigma \in \mathcal{X}$, consider the set $C(\sigma) \subseteq \mathbb{H}^2$ defined as the union of the $q$-gons centered at sites $i \in \mathcal{L}_{p,q}$ with the boundary contained in the dual lattice $\mathcal{L}_{q,p}$ and such that $\sigma(i)=+1$. The maximal connected components $C_{1}, \ldots, C_{m}, m \in \mathbb{N},$ of $C(\sigma)$ are called clusters of pluses.
\end{definition}

For an example see Figures \ref{fig:Cl} and \ref{fig:Cl2}.

\begin{figure}[!hbtp]
\centering
    \begin{minipage}{0.45\textwidth}
        \centering
        \includegraphics[width=0.74\textwidth]{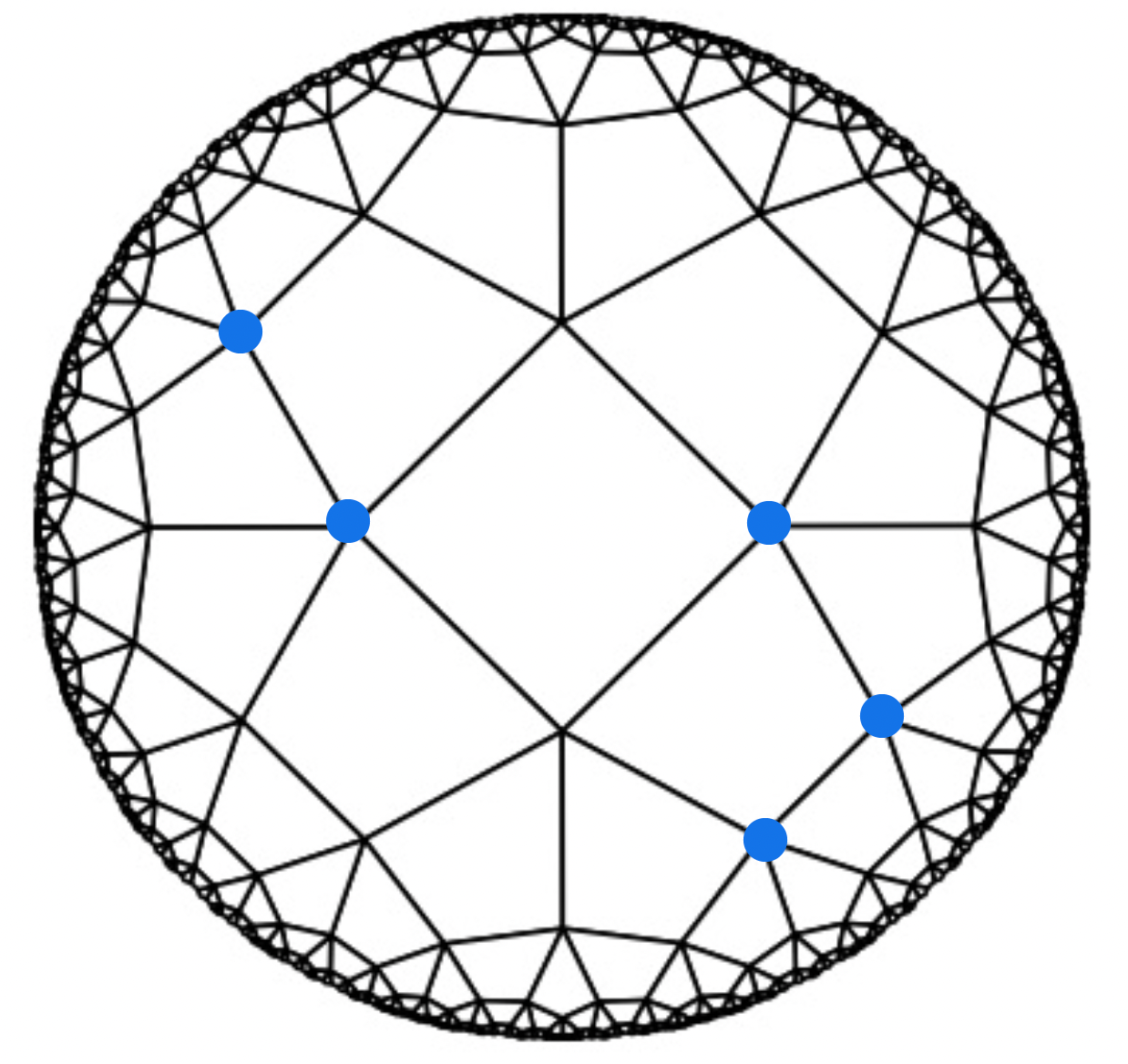} 
         \caption{Example of a configuration on $B_{5;4,5}(\textbf{o})$. Blue vertices represent plus spins and the remaining non-colored vertices $-1$ spins. }\label{fig:Cl}
    \end{minipage}\hfill
     \centering
    \begin{minipage}{0.45\textwidth}
        \centering
        \includegraphics[width=0.74\textwidth]{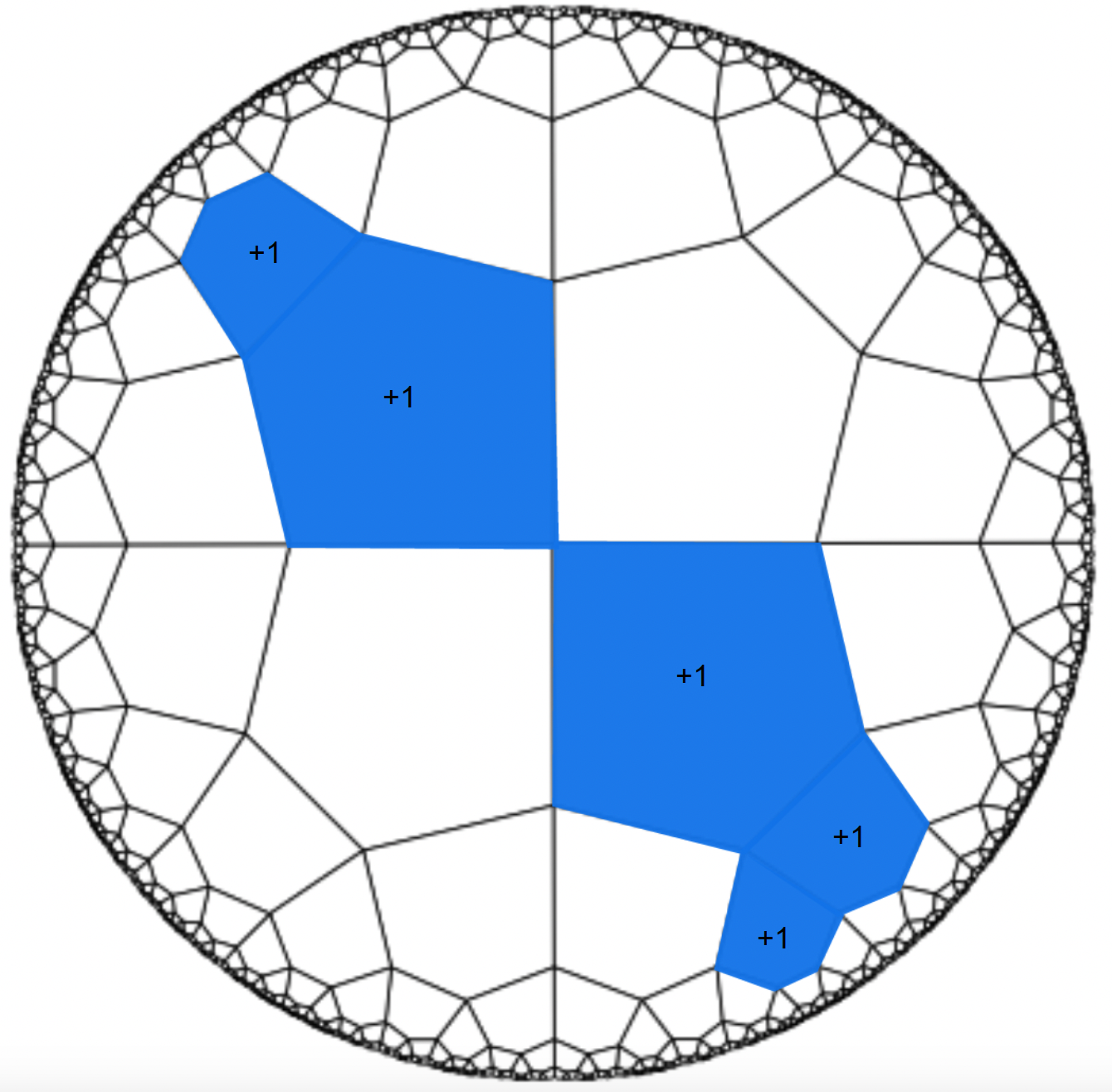} 
       \caption{Example of two corresponding clusters $C_1, C_2 \subset C(\sigma)$ for the configuration $\sigma$ depicted in Figure \ref{fig:Cl}.}\label{fig:Cl2}
    \end{minipage}\hfill
\end{figure}

Moreover, given a configuration $\sigma \in \mathcal{X}$ we denote by $\gamma(\sigma)$ its {\it Peierls contour}, that is the boundary of the clusters of pluses. 
Note that Peierls contours live on the dual lattice and are the union of 
piecewise linear curves separating spins with opposite sign in $\sigma$, see Figure \ref{fig:Peierls}. 

\begin{figure}[!htb]
\centering
        \includegraphics[width=0.4\textwidth]{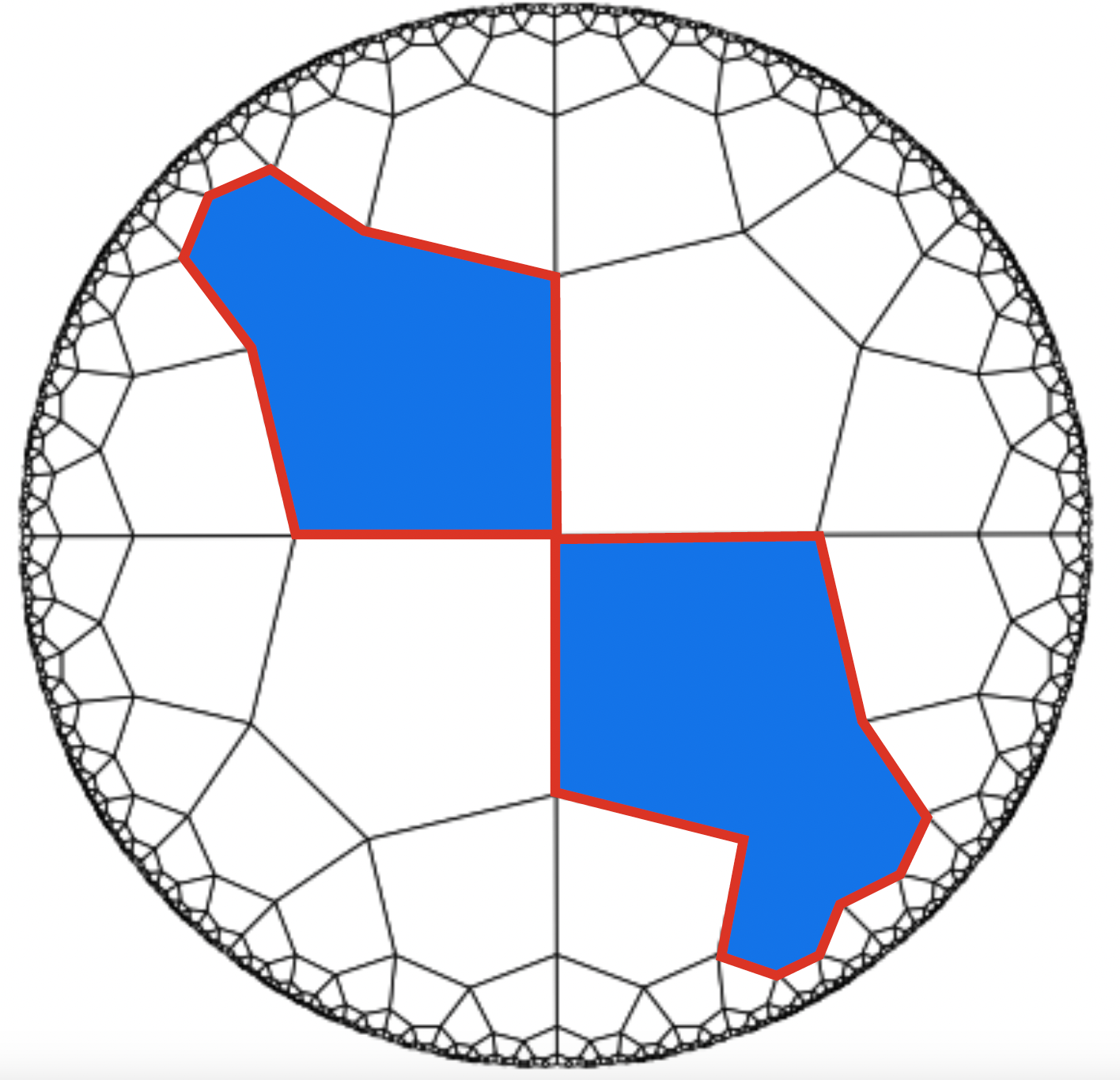}  
        \caption{An example of two Peierls contours in red for the configuration in Figure \ref{fig:Cl}. We have two polyamonds $P_1$ and $P_2$ with areas $|P_1|=2$ resp. $|P_2|=3$.}\label{fig:Peierls} 
\end{figure}

Since the gons of a cluster live naturally on the dual lattice $\mathcal{L}_{q,p}$, it is beneficial to associate clusters to geometrical shapes obtained by joining the $q$-regular gons along their edges (\emph{polyamonds}). In this way it will be possible to characterize the configurations relevant for the dynamics in terms of the area and the perimeter of the polyamonds associated to their clusters. 
Though we will consider polyamonds to study the Ising model on $\mathcal{L}_{p,q}$, the properties that we will derive may be of use to study other statistical mechanics lattice models for which the notion of clusters may be linked to that of polyamonds.

\begin{definition}\label{def:polyamond}
    A \emph{polyamond} $P \subset \mathbb{R}^2$ is a 
    finite maximally edge-connected union of $q$-gons of the lattice $\mathcal{L}_{q,p}$.
    Each gon belonging to the polyamond is called \emph{face} whereas
    the gon of $\mathcal{L}_{q,p}$ outside of $P$ are called \emph{empty face}.
\end{definition}

We remark that two faces are not connected if they
share a single point. For an example see Figure \ref{fig:Peierls}.

\begin{definition}
The \emph{area} of a polyamond $P$ is the number of its faces. We denote it by $|P|$. 
\end{definition}
\begin{definition}\label{def:perimeter_of_polyamond}
The \emph{boundary} of a polyamond $P$ is the collection of unit edges of the dual lattice $\mathcal{L}_{q,p}$ such that each edge separates a face belonging to $P$ from an empty face. The \emph{edge-perimeter} $p(P)$ of a polyamond $P$ is the cardinality of its boundary. 
\end{definition}
In other words the perimeter is given by the number of interfaces on the discrete dual lattice between the sites inside the polyamond and those outside.
If not specified differently, we will refer to the edge-perimeter simply as perimeter. 

Note that with this construction there is a bijection between clusters 
of plus spins and polyamonds. Analogously, minus spins are associated to the empty faces of the dual lattice $\mathcal{L}_{q,p}$. Thus, the number of pluses in a positive cluster corresponds to the area of the associated polyamond $P$, and its contour $\gamma$ is equivalent to the perimeter of $P$.

In the following, we define certain specific sets of vertices, which are useful in the proofs for determining the minimum-energy sets and characterizing the shapes of clusters in critical configurations. 

\begin{definition}\label{def:bmax}
Let $A \subset \mathcal{V}_{p,q}$ be any connected set of vertices (order them lexicographically $\preceq$), and set $|A|=n$. We define $B_{A, max}$ the largest ball contained in $A$ as follows. If the vertices in $A$ do not form a polygon (or a tile) then set $B_{A,max}=\emptyset$. Otherwise, there exist $M \leq n-p$ layers $L_1(x_i)$ in $A$ with middle points $x_1,\ldots,x_M$. Let
\[
\{\overline{x}, \overline{m}\} := \underset{l\in \mathbb{N}}{arg\,max} \, \underset{\{x_1,\ldots,x_M\}}{arg\, max} \left\{\bigcup_{k=0}^{l} L_k(x_i): \bigcup_{k=0}^{l} L_k(x_i) \subset A \right\}.
\]
Then $B_{A, max}$ is defined as
\[
B_{A, max}(\overline{x}) :=  \bigcup_{k=0}^{\overline{m}} L_k(\overline{x}).
\]
\end{definition}
Note that this ball is not uniquely defined. In the case that several sets $B_{A,max}(\overline{x})$ can be constructed in this way we take the ball with the smallest reference point in lexicographic order.

\begin{definition}\label{def:bmin}
Given $\overline{x}$ as in Definition \ref{def:bmax}, 
let 
\[
\overline{M}:= \underset{l\in \mathbb{N}}{arg\, min} \left \{  \bigcup_{k=0}^{l} L_k(\overline{x}): \bigcup_{k=0}^{l} L_k(\overline{x}) \supset A\right \}.
\]
We define the minimal ball containing $A$, by
\[
B_{A,Min}(\overline{x}):=  \bigcup_{k=0}^{\overline{M}} L_k(\overline{x}).
\]
\end{definition}

Furthermore, given a set of vertices $A$, we characterize the vertices $v$ in the layers of $B_{A,Min}(\overline{x})\setminus B_{A,max}(\overline{x})$ as empty if $v\notin A$ and occupied if $v\in A$. A sequence of consecutive empty, resp.~occupied, vertices in the same layer is called a \textit{strip}.

Finally, for a fixed $n \in \mathbb{N}$ and a connected set $A\subset \mathcal{V}_{p,q}$ such that $|A|=n$ we consider the set of shapes which have minimal perimeter $\mathcal{M}_N$, as proven in \cite{d2025minimal}. 

\begin{definition}\label{def:minset}
Fix $n\in \mathbb{N}$ and let $A$ be a connected set $A\subset \mathcal{V}_{p,q}$ such that $|A|=n$. We call 
\[
\mathcal{N}_e=(B_{A,Min}(\overline{x})\setminus B_{A,max}(\overline{x}))\cap A^c, \, \, \text{ resp. } \mathcal{N}_o=(B_{A,Min}(\overline{x})\setminus B_{A,max}(\overline{x}))\cap A
\]
the set of empty, resp. occupied vertices, in $B_{A,Min}(\overline{x})\setminus B_{A,max}(\overline{x})$. Define a strip $S$ of length $|\mathcal{N}_o|$ in some layer $L_K$ for $K$ large enough. Denote by $o_{max}$ the maximal possible number of vertices $v \in S$ which are  also in $I_{K;p,q}$ for $p\geq 4$. Moreover let $s_e$ denote the number of empty strips in the layers of $B_{A,Min}(\overline{x})\setminus B_{A,max}(\overline{x})$.
Then $A\in \mathcal{M}_n$ if it satisfies one of the following conditions:
\begin{itemize}
\item[(C1)] $s_e=0$ and $B_{A,Min}(\overline{x}) = B_{A,max}(\overline{x})$.
\item[(C2)] $s_e\geq 1$ and the set $\mathcal{N}_o$ contains precisely $o_{max}+(s_e-1)$ vertices $v$ such that $v\in \bigcup_{r=\overline{m}+1}^{\overline{M}} I_{r;p,q}$. 
\end{itemize}
\end{definition}

We present an example taken from \cite{d2025minimal} in Figure \ref{fig:balls}.

\begin{figure}[!hbtp]
\centering
\includegraphics[scale=0.25]{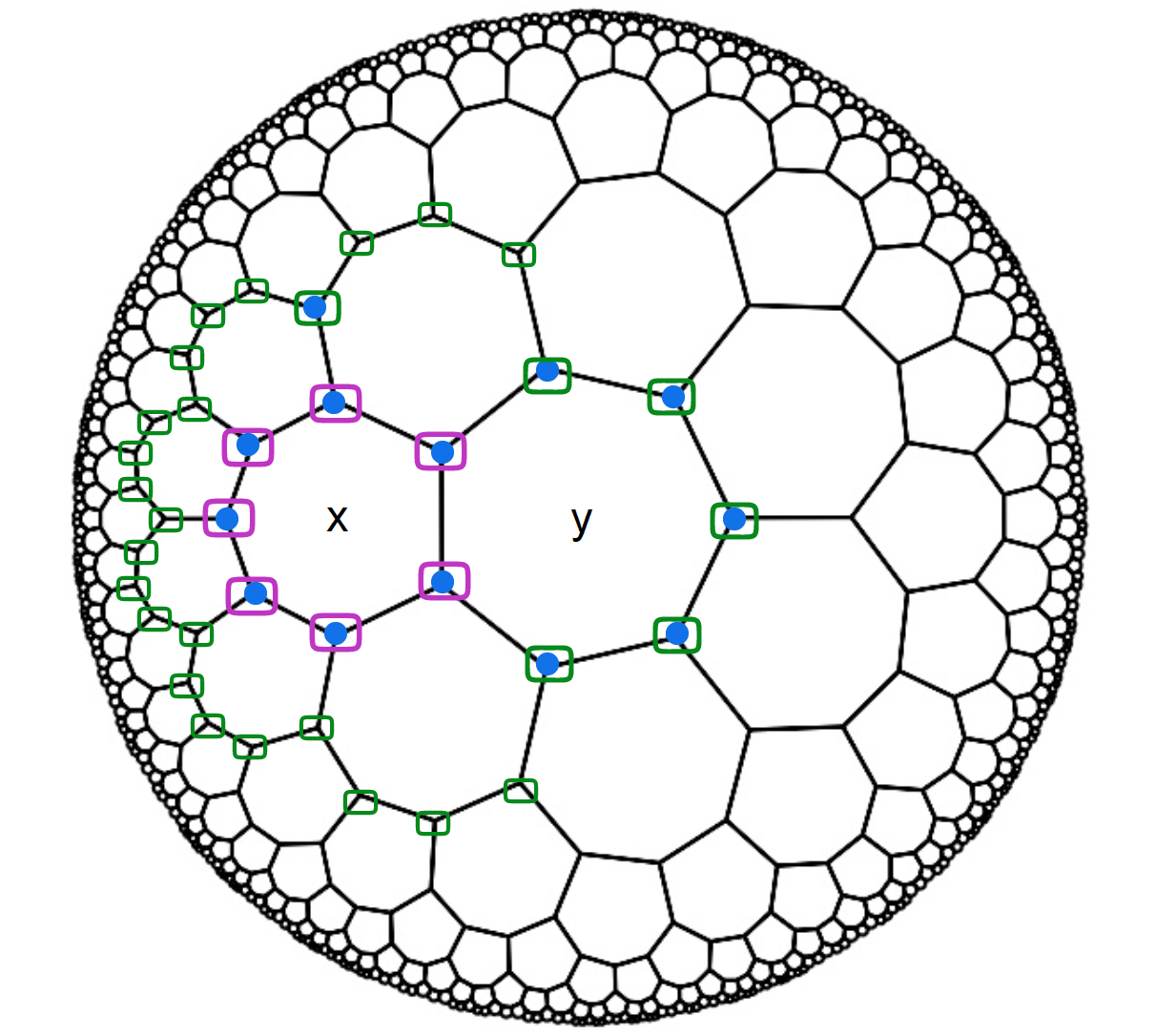}
\caption{Example of the set of connected points $A$, displayed as \textcolor{blue}{blue} points. There are two tiles present in $A$. One is centered at $x$ and one at $y$ with $x\preceq y$. The ball $B_{A,max}(x)$ is displayed by the \textcolor{magenta}{pink} circles and and the layer $B_{A,Min}(x)\setminus B_{A,max}(x)$ by the \textcolor{green}{green} circles.}\label{fig:balls}
\end{figure}

\begin{definition}\label{def:regular}
We call a polyamond \emph{standard} if its faces are centered at the vertices of a set that satisfies Definition \ref{def:minset}. 
Moreover,
a standard polyamond is \emph{regular} if its faces are centered at the vertices of the ball $B_{M;p,q}$. The total number $M$ of layers ($L_0,L_1,...,L_{M-1}$) composing the regular polyamond is called \emph{radius}.
\end{definition}
We note that a regular polyamond exists only for some value of the area $n\in\mathbb{N}$.

\begin{definition}\label{def:crit_conf}
Let $\mathscr{C}(r^*,k)$ be a standard polyamond composed of $r^*$ layers ($L_0,...,L_{r^*-1}$) and other $k$ faces attached along $L_{r^*}$, where $k\in \left \{1,...,|L_{r^*}| \right \}$. 
Among all configurations containing a single cluster of plus spins, differing with respect to $k$, let $\sigma$ be the one with maximal energy. We denote this configuration by 
$\mathcal{B}(r^*)$. 
\end{definition}
It can be observed that the number of pluses in $\mathcal{B}(r^*)$ coincides with $A^*$, where $A^*$ is given in Definition \ref{critical_radius}.

Each of these geometrical definitions and properties can be extended from polyamonds to clusters. So, for example, when we call a cluster \emph{regular cluster}, our meaning is that the cluster has the shape and the properties of a regular polyamond.

In this setting, it is immediate to see that, for each configuration $\sigma$, 
\begin{align}\label{eq:peierls_hamiltonian}
    \Delta H(\sigma)=|\gamma(\sigma)|-h N^{+}(\sigma),
\end{align}
where
\begin{align}\label{eq:number_of_pluses}
    N^{+}(\sigma)=\sum_{x \in \Lambda}\frac{\sigma(x)+1}{2},
\end{align}
represents the number of plus spins.
In this way the energy of each configuration is associated to the
area and the perimeter of a collection of associated polyamonds.

\section{Proofs of main results}\label{proofs}

We will first prove the recurrence property, then the identification of metastable states and all other theorems will follow from those.

\subsection{Proof of Proposition~\ref{teoRP}}\label{sec:recurrence}
Assume that $h \in (h^*_1,h^*_2)$.

We will prove that each configuration in  $\mathcal{X} \setminus \{\textbf{-1}, \textbf{+1}\}$ has a stability level smaller or equal to $K^*=\max \{2(h-q+4), K^*(h)\}$, where $K^*(h)$ is a function of $h,p,q$ that will be defined precisely in \eqref{stabR1}. With this result, we may apply \cite[Theorem 3.1]{MNOS04} to derive the claim of Proposition~\ref{teoRP}. 

The strategy of the proof is as follows. In \textbf{Step 1},  we will prove that the stability level of a configuration $\sigma \in \mathcal{X} \setminus \{\textbf{-1}, \textbf{+1}\}$ containing at least a regular cluster, is smaller than $K^*(h)$. Then, in \textbf{Step 2}, we will show that the configurations that contains at least a cluster different from a regular cluster, has a stability level smaller than $2(h-q+4)$.

\subsubsection*{Step 1: Recurrence for balls.}

We consider the set of configurations $\mathscr{R} \subset \mathcal{X}$ that contain only regular clusters of pluses. We partition $\mathscr{R} \setminus \{\textbf{+1}^{(-)}\} := R_1 \cup R_2$ in two sets:
\begin{itemize}
\item $R_1$ is composed by the configurations that contain only regular clusters of pluses with radius $n>r^*$, 
\item $R_2$ is composed by the configurations that contain only regular clusters of pluses with at least one of them with radius $1 \leq n \leq r^*$, 
\end{itemize}

We will see that regular clusters in $R_1$ have the tendency to grow with high probability, whereas regular clusters in $R_2$ have the tendency to shrink.

\paragraph{Analysis of the configurations in $R_1$.}  Given a configuration $\sigma$ in $R_1$, we consider one of the regular clusters of pluses $C^+$ with radius $n > r^*$.
    
    We construct a path $\omega\in \Theta(\sigma,\mathcal{I}_{\sigma} \cap (R_1 \cup \{\textbf{+1}^{(-)}\}))$ that
    adds an external layer of pluses to the configuration. In particular, we flip all the minuses in $L_{n+1}$.
    Starting from $\sigma\equiv \omega_0\in R_1$, we will define $\omega_1$ as follows. 
    Following the order that maximizes the number of connected vertices $I_{n+1;p,q}$, we denote by $j_1,\ldots,j_r$ the sites belonging to the layer $L_{n+1}$ with $r=|L_{n+1}|$. Thus, the first vertices are such that $j_1 \in I_{n+1;p,q}$, $j_2,\ldots,j_{p-2} \in E_{n+1;p,q}$ and $j_{p-1}\in I_{n+1;p,q}$. 
   
    Consider the first minus in $j_1$ and flip it into a plus, i.e., $\omega_1:=\omega_0^{(j_1)}$. Define $\omega_k:=\omega_{k-1}^{(j_k)}$, for $k=2,\ldots,r$. By a direct computation, we obtain for $k=2,\ldots,r-1$
    \begin{align}
        &H(\omega_1)=H(\omega_0)-h+q-2, \notag \\
        &H(\omega_{k})=
        \begin{cases}
            H(\omega_{k-1})-h+q-2 &\qquad \text{ if } j_k\in E_{n+1;p,q}  \notag \\
            H(\omega_{k-1})-h+q-4 &\qquad \text{ if } j_k\in I_{n+1;p,q} 
        \end{cases} \notag \\
        &H(\omega_r)=H(\omega_{r-1})-h+q-4.
    \end{align}
    Then, by noting that $j_1\in I_{n+1;p,q}$ and $j_r \in E_{n+1;p,q}$, we have
    \begin{align}\label{energy_omegar}
H(\omega_r)&=H(\omega_0)+(-h+q-4)|I_{n+1;p,q}|+(-h+q-2)|E_{n+1;p,q}| \notag \\
&=H(\omega_0)+(q-2-h)|L_{n+1}|-2|I_{n+1;p,q}|<H(\omega_0),
    \end{align}
where the last inequality follows from the assumption that $n > r^*$. Indeed
\begin{align}
    |L_{n+1}|(q-2-h)&-2|I_{n+1;p,q}|
    =c_{p,q} (a_- \lambda_-^{n+1} +a_+ \lambda_+^{n+1})(q-2-h)-4 c_{p,q} \sqrt{q-2} (\lambda_+^{n+1}-\lambda_-^{n+1}) \notag \\
    &=c_{p,q}\lambda_-^{n+1} \left [ \left (a_-+a_+ \frac{\lambda_+^{n+1}}{\lambda_-^{n+1}} \right) (q-2-h)-4 \sqrt{q-2} \left (\frac{\lambda_+^{n+1}}{\lambda_-^{n+1}} -1 \right )\right ] \notag \\
    &=c_{p,q}\lambda_-^{n+1} \left [ \left (a_+(q-2-h)-4 \sqrt{q-2} \right )\frac{\lambda_+^{n+1}}{\lambda_-^{n+1}}  +4 \sqrt{q-2} +a_-(q-2-h)\right ],
\end{align}
and the coefficient of the ratio is negative for $h> h^*_1$.
Thus, we have that $H(\omega_r)<H(\omega_0)$ if and only if
\begin{align}
    \frac{\lambda_+^{n+1}}{\lambda_-^{n+1}}>\frac{4 \sqrt{q-2}+a_- (q-2-h)}{4 \sqrt{q-2}-a_+(q-2-h)}
\end{align}
and this inequality is equivalent to require $n > r^*$ for $n \neq N$ where $N$ denoted the last layer of the considered lattice $\Lambda$. 

In order to compute $V_{\sigma}$ we note that 
Since $q\geq 4$, by construction two vertices in $I_{n+1;p,q}$ are separated by $p-3$ or $p-4$ vertices in $E_{n+1;p,q}$. In particular, we define the path $\omega$ that follows a clockwise order and flips the minus spins adjacent to the cluster of pluses starting from a vertex in $I_{n+1;p,q}$, which is separated from the next vertex in $I_{n+1;p,q}$ by $p-4$ vertices in $E_{n+1;p,q}$. Furthermore, the following $q-2$ pairs of vertices in $I_{n+1;p,q}$ are separated by $p-4$ vertices in $E_{n+1;p,q}$. Then, we have another pair of vertices in $I_{n+1;p,q}$ separated by $p-3$ vertices in $E_{n+1;p,q}$.

\begin{align}\label{schema}
    \underbrace{I \,\,\,\,\,\,\, (p-4) E \,\,\,\,\,\,\, } I \underbrace{\,\,\,\,\,\,\, (p-3) E \,\,\,\,\,\,\, I \,\,\,\,\,\,\, (p-3) E \,\,\,\,\,\,\, I \,\,\,\,\,\,\, (p-3) E \,\,\,\,\,\,\, I \,\,\,\,\,\,\, ... \,\,\,\,\,\,\,\,}_{q-4} \,I\underbrace{\,\,\,\,\,\,\, (p-4) E \,\,\,\,\,\,\, I}
\end{align}

Thus, the first two pluses in $I_{n+1;p,q}$ that we flip are divided by $p-4$ vertices in $E_{n+1;p,q}$. We obtain
\begin{align}
    &H(\omega_1)=H(\omega_0)+q-2-h, \notag \\
    &H(\omega_{i+1})=H(\omega_0)+q-2-h+i (q-2-h) &&\qquad\text{ for } i=1,...,p-4. \notag \\
    &H(\omega_{p-2})=H(\omega_0)+(p-3)(q-2-h)+q-4-h \notag
\end{align}

Note that $H(\omega_{k})-H(\omega_{0})>0$ for each $k=1,...,p-2$  since $h< h^*_2<q-2-\frac{2}{p-2}$. 
In particular 
\[
\max_{k=0,...,p-2} H(\sigma_k)=H(\sigma_{p-3}).
\]
We distinguish two cases: 
\begin{itemize}
    \item[(i)] $q>4$. In this case,
the second vertex of $I_{n+1;p,q}$, i.e. $j_{p-1}$, is separated from the third one, $j_{2p-5}$, by $p-3$ vertices in $E_{n+1;p,q}$, then
\begin{align}
    &H(\omega_{(p-2)+i})=H(\omega_{p-2})+i (q-2-h) &&\qquad\text{ for } i=1,...,p-3, \notag \\
    &H(\omega_{2(p-2)})=H(\omega_{p-2})+(p-3)(q-2-h)+q-4-h \notag 
\end{align}
In this case we note that $H(\omega_{k})-H(\omega_{p-2})>0$ for each $k=(p-2)+1,...,2(p-2)$ since $h< h^*_2<q-2-\frac{2}{p-2}$. In particular,
\[
\min_{k=0,...,2(p-2)} H(\omega_k)=H(\omega_0)
\]
and
\[
\max_{k=p-1,...,2(p-2)} H(\omega_k)=H(\omega_{2p-5}).
\]
We compare the two maxima and we find that $H(\omega_{2p-5})>H(\omega_{p-3})$.
See Figure \ref{fig:examples} for a schematic representation of the energy landscape for $p,q=5$ and two different values of $h$.
Then we iterate this procedure $q-5$ more times, for a total of $q-4$ times as represented in the representation \eqref{schema}. 
After performing the preceding iterations, the vertices in $E_{n;p,q}$ that lie between two consecutive vertices of $I_{n;p,q}$ are 
$p-4$. Subsequently, they revert to $p-3$. It follows that, when considering all the values of the maxima, they do not increase or decrease monotonically, not even when restricting to those appearing along the path segments that recur recursively (see Appendix \ref{sec:appendix} for two examples). The value of the maximum depends on the chosen parameter 
$h$. A similar phenomenon is observed in \cite{apollonio2022metastability}, but in that case the maximum oscillates between only two values determined by $h$, whereas here multiple values occur.
Thus, let $\omega_{\tilde h}$ be the configuration with maximal energy along this path, then 
\begin{align}\label{stabR1}
    V_\sigma&=H(\omega_{\tilde h})-H(\omega_{0}):=K^*(h).
\end{align}

\item[(ii)] $q=4$. In this case, we proceed in a similar manner to the previous one, but taking into account that each pair of vertices in $I_{n+1;p.q}$ is separated by $p-3$ vertices from $E_{n+1,p,q}$.
\end{itemize}

\paragraph{Analysis of the configurations in $R_2$.} Given a configuration $\sigma$ in $R_2$, we consider one of the regular cluster of pluses $C^+$ with radius $n \leq r^*$. By Definition \ref{def:regular}, we have that $C^+=\bigcup_{k=0}^{n-1} L_k$. First we assume $n\geq 3$ and at the end  we analyze the simple case $n\leq 2$.
    We construct a path $\omega\in \Theta(\sigma,\mathcal{I}_{\sigma} \cap R_2 )$ that
    dismantles the external layer. In particular, we flip all pluses in $L_n$ by following the time-reversal of the path described in the previous paragraph for a configuration in $R_1$. Thus, let $j_1,...,j_r$ be the vertices in $L_n$ ordered in the reverse order compared to the previous case, so $j_1,...,j_{p-4} \in E_{n;p,q}$, $j_{p-3}\in I_{n;p,q}$ and so on. We define $\omega_k:=\omega_{k-1}^{(j_k)}$, for $k=1,...,r$. By a direct computation, we obtain for $k=2,...r-1$
    \begin{align}
        &H(\omega_1)=H(\omega_0)+h+(-(q-2)+2), \notag \\
        &H(\omega_{k})=
        \begin{cases}
            H(\omega_{k-1})+h+(-(q-1)+1) &\qquad \text{ if } j_k\in E_{n;p,q}  \notag \\
            H(\omega_{k-1})+h+(-(q-2)+2) &\qquad \text{ if } j_k\in I_{n;p,q}. 
        \end{cases} \notag \\
        &H(\omega_r)=H(\omega_{r-1})+h+(-(q-1)+1)
    \end{align}
    Then, noting that $j_1\in E_{n;p,q}$ and $j_r\in I_{n;p,q}$, we compute
    \begin{align}\label{omegar_omega0_shrink}
H(\omega_r)&=H(\omega_0)+|L_n|h+(-(q-2)+2)|I_{n;p,q}|+(-(q-1)+1)|E_{n;p,q}| \notag \\
&=H(\omega_0)-|L_n|(q-2-h)+2|I_{n;p,q}|<H(\omega_0),
    \end{align}
where the last inequality follows from the assumption on $n \leq r^*$ and $h>h^*_1$. 
By using the time-reversal of the path described in the previous case, we compute $V_{\sigma}$ and we find 
\begin{align}\label{stab_R2}
V_\sigma=H(\omega_1)-H(\omega_{\tilde h})=K^*(h).
\end{align}

Denote by 
\[
\mathcal{X}^+_{r,k} = \{\sigma \in \mathcal{X} \, | \, \sigma_{B_{r;p,q}}\equiv +1, \text{ and } \exists v_1,\ldots,v_k \in L_{r} \text{ such that } \sigma(v_i)=+1 \text{ for all } i \}.
\]
the space of configurations with all pluses on a ball of radius $r$ and in precisely $k$ vertices in the subsequent layer. Fix $h \in (h^*_1, h^*_2)$ and set
$r=r^*$. Then the critical droplet is defined as, 
\begin{equation}\label{eq:crit_drop}
\mathcal{B}_{r^*;p,q}(o)\cup S_{crit}(h) = \left \{\text{supp}(\sigma)\cap\{i:\sigma(i)=+1\} \, | \, \sigma \in \underset{\sigma\in \mathcal{X}^+_{r^*,k}, \, \, \\k=1,\ldots,|L_{r^*}|}{argmax} H(\sigma) \right\}.
\end{equation}
In other words when $r=r^*$, the configuration $\omega_{\tilde h}$ contains the critical droplet, i.e. $\omega_{\tilde h} \equiv \mathcal{B}(r^*)$, see Definition \ref{def:minset} with $s_e\neq 0$ and Definition \ref{def:crit_conf}.

To conclude, consider now $n=2$. We construct a path $\omega\in \Theta(\sigma,\mathcal{I}_{\sigma} \cap R_2 )$ that dismantles the layer $L_1$. In particular, we note that by construction all the vertices in $I_{1;p,q}$ are separated by $p-3$ vertices in $E_{1;p,q}$. Thus, we flip all the pluses in $L_1$ starting from the vertices in $E_{1;p,q}$ as follows. Let $j_1,...,j_{|L_1|}$ be the ordered sites in $L_1$ such that $j_1,...,j_{p-3}\in E_{1;p,q}$. We define $\omega_k:=\omega_{k-1}^{(j_k)}$, for $k=1,...,r$. By a direct computation, we obtain for $k=2,...r-1$
    \begin{align}
        &H(\omega_1)=H(\omega_0)+h+(-(q-2)+2), \notag \\
        &H(\omega_{k})=
        \begin{cases}
            H(\omega_{k-1})+h+(-(q-1)+1) &\qquad \text{ if } j_k\in E_{1;p,q}  \notag \\
            H(\omega_{k-1})+h+(-(q-2)+2) &\qquad \text{ if } j_k\in I_{1;p,q}. 
        \end{cases} \notag \\
        &H(\omega_r)=H(\omega_{r-1})+h+(-(q-1)+1)
    \end{align}
Thus,
    \begin{align}
H(\omega_r)&=H(\omega_0)+|L_1|h+(-(q-2)+2)|I_{1;p,q}|+(-(q-1)+1)|E_{1;p,q}| \notag \\
&=H(\omega_0)-|L_1|(q-2-h)+2|I_{1;p,q}|<H(\omega_0).
    \end{align}
We argue as above and we find $V_\sigma \leq h-q+4$. 

In case $n=1$, the cluster of pluses is composed only by $|L_0|=p$ pluses that we flip into minuses. In particular, let $j_1,...,j_p$ be the ordered sites in $L_0$ and we define $\omega_k:=\omega_{k-1}^{(j_k)}$, for $k=1,...,p$. 
By a direct computation, we obtain
    \begin{align}
        &H(\omega_1)=H(\omega_0)+h+(-(q-2)+2), \notag \\
        &H(\omega_{k})=
            H(\omega_{k-1})+h+(-(q-1)+1) \qquad \text{ for } k=2,...,p-1. \notag \\
        &H(\omega_p)=H(\omega_{p-1})+h-q
    \end{align}
Thus,
\begin{align}
    H(\omega_p)=H(\omega_0)+p(h-q+2)+2<H(\omega_0)
\end{align}
since $h<h^*_2 <q-2-2/p$. Moreover $V_\sigma \leq h-q+4$ since the rest of the path is downhill.

In conclusion, by using the last computation and Equations \eqref{stabR1} and \eqref{stab_R2}, we have that the stability level of $\mathscr{R}$ is $V_{\mathscr{R}} \leq  K^*(h)$.

\subsubsection*{Step 2: Recurrence for other configurations.}
In this section, we will prove that a configuration $\sigma \not \in \{\textbf{-1},\textbf{+1}\}$, containing at least a cluster of pluses $C$ different from a regular cluster, has a stability level not greater than $2(h-q+4)$. To do this, we create a path $\omega$ from $\sigma$ to $\eta$ such that 
\begin{align}
    H(\eta)<H(\sigma) \qquad \text{ and } \qquad \Phi (\sigma,\eta)-H(\sigma) \leq 2(h-q+4).
\end{align}

We construct $\omega=(\sigma_0,...,\sigma_n)$ such that $\sigma_0:=\sigma$ and $\sigma_n:=\eta$ in the following way. 
Recalling Definition \ref{def:bmin}, consider the minimal union of layers $B_{C,Min}$ circumscribing the cluster $C$. We observe that there is at least a minus spin in $B_{A,Min}$, otherwise $C$ is a regular cluster and this is a contradiction. We also note that $B_{C,Min}$ may contain some plus spins belonging to other clusters. We analyze the values of the spins in $B_{C,Min}$ and we distinguish two cases:
\begin{itemize}
    \item[(1)] There exists a vertex $v \in B_{C,Min}$ such that $\sigma_0(v)=+1$ is connected with at most a plus of the cluster $C$.
    \item[(2)] All pluses in $B_{C,Min}$ are connected with at least other two pluses of $C$.
\end{itemize}

We start with case (1). Consider the configuration $\sigma_1$ obtained starting from $\sigma_0$ by flipping the plus in $v$ into minus, i.e. $(\sigma_1)_{\Lambda \setminus \{v\}}=(\sigma_0)_{\Lambda \setminus \{v\}}$ and $\sigma_1(v) \neq \sigma_0(v)$. Thus, $H(\sigma_1) \leq H(\sigma_0)+h-q+2$ by assumption on $v$. Since $h<h^*_2<q-2-\frac{2}{p-2}$, we conclude $H(\sigma_1)<H(\sigma_0)$ and $V_{\sigma_0}=0$.

Next, we study case (2). Let $L_k \subset B_{C,Min}$ be the smallest layer containing a minus spin, $k\leq \overline{M}$ where $\overline{M}$ is the number of layers in $B_{C,Min}$. We note that this spin exists, otherwise the configuration contains a regular cluster. We distinguish three cases:
\begin{itemize}
\item[(2.a)] If $k > r^*$, then we construct a path as in case of configurations of  $R_1$ by flipping the minuses in $L_{\overline{M}}$. The presence of the plus spins in the outer layer promotes the tendency of the system to follow the described path and $V_{\sigma} \leq h-q+4$. 
\item[(2.b)] If $k \leq r^*$ and $k=\overline{M}$, then we construct a path as in the case treating configurations in $R_2$ by flipping the pluses in $L_{\overline{M}}$. In this case, the presence of the minuses in $L_{\overline{M}}$ favors the evolution of the system along the described path. Thus, for each value of $M$, $V_{\sigma} \leq h-q+4$.
\item[(2.c)] If $k \leq r^*$ and $k<\overline{M}$.

First, we suppose that the vertices in $L_{k+1}\cup L_{k-1}$ connected with a minus in $L_k$ have plus spin. (We note that these pluses are not necessary in $C$, they might belong to another cluster of pluses.)
In this case, we construct $\sigma_1$ starting from $\sigma_0$ by flipping this minus. By a direct computation, we obtain $H(\sigma_1)\leq H(\sigma_0)-h-q+4<H(\sigma_0)$ since $h>0$ and $q \geq 4$. Thus, we can conclude $V_{\sigma_0}=0$.

In the other case, we consider the minus strips in $B_{C,Min}$. 
Let $s_1,...,s_\ell$ be these strips, for some $\ell \in \mathbb{N}$. We note that 
$\ell \not =0$, otherwise $C$ is a regular cluster that is a contradiction. 
We subdivide further into two cases:
\begin{itemize}
\item[(i)] There exists at least a plus strips $s \in L_{\overline{M}}$ with length $|s| \leq |L_{r^*}|$. In this case we construct a path by flipping the plus spins in $s$ as in case of $R_2$. Thus, also in this case we obtain $V_\sigma \leq h-q+4$.
\item[(ii)] All the plus strips in $L_M$ have length greater than $|L_{r^*}|$. We denote by $s$ one of these strips and we construct a path $\omega$ by flipping some pluses in $s$ and simultaneously by swapping some minuses in $L_k$. In the following, we describe $\omega$ in detail. Let $s'$ be one of the minus strips in $L_k$. We observe that $|s'| \leq |L_{r^*}|$ since $k \leq r^*$. Thus, $|s'|<|s|$ and we define the path $\omega$ by flipping $|s'|$ pluses of $s$ and by flipping all the minuses of $s'$. We call $\eta$ the last configuration and we obtain
\begin{align}
    H(\eta)=H(\sigma)+h |s'| - n_I^s (q-4) - n_E^s (q-2)-h |s'|+n_I^{s'} (q-4) + n_E^{s'} (q-2)
\end{align}
where $n_I^s$ (resp. $n_E^{s}$) is the number of plus spins in $s \cap I_{M;p,q}$ (resp. $s \cap E_{M;p,q}$) that we flip into minus, and $n_I^{s'}$ (resp. $n_E^{s'}$) is the total number of minus spins in $s' \cap I_{k;p,q}$ (resp. $s' \cap E_{k;p,q}$). Since by construction $n_I^s \in \{n_I^{s'},n_I^{s'}+1\}$ and $|s'|=n_I^s+n_E^s=n_I^{s'}+n_E^{s'}$, then we have $H(\eta) \leq H(\sigma)$ and $V_\sigma \leq 2(h-q+4)$ by arguing as in case $R_1$ and $R_2$ and estimating with the upper value of the stability level along these two paths. Moreover, if $H(\eta) < H(\sigma)$ we conclude. Otherwise, we analyze the configuration $\eta$, in particular we focus on the remaining pluses in $s$. If there exists a plus spin that satisfies the assumption of (1), we conclude as above; otherwise we iterate this step (ii) by considering another strip of minuses in $L_k$ (if there exists) or in $L_{k+1}$ if $L_k$ contains all pluses. We conclude by noting that when we flip the last pluses in $s$ the energy decreases. Thus, in all of these cases $V_\sigma \leq 2(h-q+4)$.

\end{itemize}

\end{itemize}

\begin{remark}
We note that $K^*(h) \geq (p-3)(q-2-h)$, see the first part of the path in \textbf{Step 1}. By assumption on $h$, we have $h<h_2^*<q-2-2/(p-2)$, thus
\begin{align}
    K^*(h) \geq (p-3)(q-2-h) >h-q-4.
\end{align}
We may conclude that $K^*=K^*(h)$.
\end{remark}

\subsection{Proof of Theorem \ref{Identification}}\label{bound}

In this section, we prove that the stability level of $\textbf{-1}$ is $\Gamma^{p,q}$ and, since $\Gamma^{p,q}>K^*$, then the homogeneous state $\textbf{-1}$ is a metastable state.
We will prove in two steps, that
\begin{equation}
 \Phi(\textbf{-1},\textbf{+1})-H(\textbf{-1})=\Gamma^{p,q}.
\end{equation}

First, we create a \emph{reference path} and estimate the upper bound for the maximal stability level, in a second step we find the lower bound.

The claim of the theorem can be derived from the previous estimate and the following argument.

The first assumption of \cite[Theorem 2.4]{CNrelax13} is satisfied for the choice of $A=\{\textbf{-1}\}$ and $a=\Gamma^{p,q}$. The second assumption of \cite[Theorem 2.4]{CNrelax13} is satisfied (see Section \ref{sec:recurrence}), since either $\mathcal{X}\setminus \{\textbf{-1}, \textbf{+1} \}=\emptyset$ or
$V_\sigma<\Gamma^{p,q}$ for all $\sigma\in \mathcal{X}\setminus \{\textbf{-1}, \textbf{+1} \}$.
Thus, we can conclude by applying \cite[Theorem 2.4]{CNrelax13}.

\subsubsection*{Step 1: Upper bound}\label{referencepath}
In this Section we define the \emph{reference path} $\omega^\mathrm{r}$ as a sequence of configurations from $\textbf{-1}$ to $\textbf{+1}$ that are increasing clusters \emph{as close as possible to regular shape}. 
We first describe intuitively this path: starting from $\textbf{-1}$ we flip one of the $p$ nearest spins at the origin and then by following the clockwise order we flip the other $p-1$ minuses close to the origin. After that, the configuration is a regular cluster composed of only one layer $L_0$. Then, we flip the minuses in $L_1$ starting from the spin in a vertex of $I_{1;p,q}$. In this way we obtain a regular cluster with radius $2$.
We iterate this procedure until the regular cluster of pluses fills all $\Lambda$, i.e. it reaches the configuration $\textbf{+1}$. 

In the following we describe the reference path in more details and we prove that along this path $\Phi(\omega^\mathrm{r})-H(\textbf{-1})=\Gamma^{p,q}$. This implies that
\begin{equation}
\label{upper}
\Phi(\textbf{-1},\textbf{+1})-H(\textbf{-1})\leq \Gamma^{p,q}.
\end{equation}

We denote by $\sigma_{n}$ the configuration that contains 
a regular cluster with radius $n$, i.e. the configuration containing a regular cluster with $n$ layers of pluses centered in the origin. 

We start by choosing one of the $p$ minuses close to the origin. We flip it into plus with an energy cost equal to $q-h$.
Then, following the clockwise order, we flip the other remaining $p-1$ minuses, one after the other,
with an energy cost equal to $q-2-h$ for the first $p-2$ steps and equal to $q-4-h$ for the last step.
Thus the total energy cost to construct $\sigma_{1}$, i.e. to form a regular cluster with only one layer, is $p(q-2-h)$.

Next, we consider a minus in $I_{1;p,q}$ and we flip it into plus. Then, starting from this plus and following the clockwise order we flip all the minuses in $L_1$ and we obtain $\sigma_2$.
We grow up this regular cluster by flipping the minuses in $L_2$ starting from a minus in $I_{2;p,q}$ and following the clockwise order. In this way we construct $\sigma_3$. 

This growth mechanism can be iterated, 
until the regular cluster invades $\Lambda$ and the configuration $\textbf{+1}$ is reached.

In order to compute the height of this path, we first evaluate 
the height of the portion of the path connecting $\textbf{-1}$ to a general
configuration $\sigma_n$. We have 
\begin{align}
  H(\sigma_n)& =H(\textbf{-1})-h\sum_{k=0}^{n-1}|L_k|+|I_{n:p,q}|.
\end{align}
Thus, we compare the energy of $\sigma_n$ with the energy of $\sigma_{n+1}$ and we obtain
\begin{align}
H(\sigma_{n+1})- H(\sigma_n)&=|I_{n+1:p,q}|-|I_{n;p,q}|-h|L_{n}| \notag \\
&=(q-3)|I_{n;p,q}|+(q-2)|E_{n;p,q}|-|I_{n;p,q}|-h|L_{n}| \notag \\
&=(q-2-h)|L_{n}|-2|I_{n;p,q}| \notag \\
&=c_{p,q} (a_- \lambda_-^{n} + a_+ \lambda_+^{n})(q-2-h)-4c_{p,q} \sqrt{q-2} (\lambda_+^{n}-\lambda_-^{n}).
\end{align}
Recalling the definition of critical radius Definition \ref{critical_radius}, we note that this function is increasing in $n$ if $n \in [0, r^*]$ and it is decreasing for $n \in [r^*+1, N)$, that is
\begin{align}
&    H(\textbf{-1})<H(\sigma_1)<H(\sigma_2)<\ldots <H(\sigma_{r^*-1})<H(\sigma_{r^*})  \text{ and } \notag\\
& H(\textbf{+1}) <  H(\sigma_{N-1})< \ldots H(\sigma_{r^*+1}) < H(\sigma_{r^*}). \notag
\end{align}
Thus, the height of the reference path $\Phi(\omega^\mathrm{r})$
is equal to 
$\Phi(\sigma_{r^*},\sigma_{r^*+1})$, since the energy cost between two configurations $\sigma_i$ and $\sigma_{i+1}$ is the same for each $i=1,...,N-1$ (see Section \ref{sec:recurrence} \textbf{Step 1} for details).  
Analogously to the proof of the recurrence property for regular clusters, we have 
    \begin{align}
        &\Phi(\sigma_{r^*},\sigma_{r^*+1})=H(\sigma_{r^*})+K^*.
    \end{align}
    Thus, 
    \begin{align}
        \Phi(\omega^\mathrm{r})-H(\textbf{-1})
        =&H(\sigma_{r^*})+K^*-H(\textbf{-1}) \notag \\
        =&c_{p,q} \left [(a_- \lambda_-^{r^*} + a_+ \lambda_+^{r^*})(q-2-h)-4 \sqrt{q-2} (\lambda_+^{r^*}-\lambda_-^{r^*}) \right ]+K^* \notag \\
        =&\Gamma^{p,q}
    \end{align}
    where last equality follows from the Definition \ref{def:gamma} of $\Gamma^{p,q}$.

\subsubsection*{Step 2: Lower bound}
Given $\sigma \in \mathcal{X}$, we recall \eqref{eq:number_of_pluses} for the number of plus spins in $\sigma$. 
For $n$ integer, $0 \leq n \leq |\Lambda|$, we introduce the following set
\begin{equation}\label{Vn}
    \nu_n:=\{\sigma \in \mathcal{X} | N^+(\sigma)=n\},
\end{equation}
namely $\nu_n$ is the set of configurations with a number of plus spins fixed at the value $n$. 

Any path $\omega$ from $\textbf{-1}$ to $\textbf{+1}$ crosses a configuration in each manifold for $n=0,...,|\Lambda|$. 
In particular, any path $\omega:\textbf{-1} \to \textbf{+1}$ crosses the manifold $\nu_{A^*}$ in any configuration, see Definitions \ref{critical_radius} and \ref{def:crit_conf} for the value of the critical area $A^*$.
By \cite{d2025minimal} we know that one of the configurations that contain a cluster of pluses with minimal perimeter is a configuration that contains a cluster of pluses $\mathscr{C} (r^*)$, see Definition \ref{def:crit_conf} In particular, $H(\mathcal{B}(r^*))=\Gamma^{p,q}$ and any other configurations belonging to $\nu_{A^*}$ has energy greater than $\Gamma^{p,q}$, since the number of pluses is the same and its perimeter is greater than that of $\mathscr{C} (r^*)$ by \cite{d2025minimal}. In particular, if a configuration in $\nu_{A^*}$ contains a non-regular cluster then its energy is strictly greater than $\Gamma^{p,q}$. 

Thus, $\Phi(\textbf{-1}, \textbf{+1})-H(\textbf{-1}) \geq \Gamma^{p,q}$. 

\subsection{Proof Theorems \ref{thm:transitiontime}, \ref{teotime'} and \ref{TMIX}}\label{sec:other_theorems}

\begin{proof}[Proof of Theorem~\ref{thm:transitiontime}] Thanks to Theorem \ref{Identification}, we can conclude by applying \cite[Theorem 4.1]{MNOS04} with $\eta_0=\{\textbf{-1}\}$ and $\Gamma=\Gamma^{p,q}$.
\end{proof}
\begin{proof}[Proof of Theorem~\ref{teotime'}]
By Proposition~\ref{teoRP}, the assumptions of \cite[Theorem 4.15]{MNOS04} are verified taking $\eta_0=\{\textbf{-1}\}$ and $T'_{\beta}=e^{\beta(V^*+\epsilon)}$. Then \eqref{Ptime'} and \eqref{Etime'} follow from \cite[Theorem 4.15]{MNOS04}. 
\end{proof}
\begin{proof}[Proof of Theorem~\ref{TMIX}]
Thanks to Theorem \ref{Identification}, we have that the maximal stability level is $\Gamma^{p,q}$, then we get the result by \cite[Proposition 3.24 and Example 3]{nardi2016hitting}.
\end{proof}

\section{Discussion on the other regions}\label{sec:other_regions}

In the following we will study the behavior of the Ising system for $h \notin (h^*_1, h^*_2)$. We will discuss the regions \textbf{Region I}: $h < h^*_1$ and \textbf{Region II}: $h \geq h^*_2$ separately. First we will prove a useful lemma. 

\begin{lemma}
Let $h>0$, then we have that 
\begin{align}
\min_{\sigma \in \mathcal{X}} \Delta H(\sigma)=
    \begin{cases}
        \Delta H(\textbf{-1}) &\qquad \text{ if } h\leq h^*_1,\\
        \Delta H(\textbf{+1}) &\qquad \text{ if } h\geq h^*_1.
    \end{cases}
\end{align}
\end{lemma}
We note that if $h=h^*_1$ then we have two global minima $\textbf{-1}, \textbf{+1}$ and hence no metastability.
\begin{proof}
We compute the energy difference of the configurations $\textbf{-1}, \textbf{+1}, \textbf{-1}^{(+)}, \textbf{+1}^{(-)}$ with $\textbf{-1}$. It yields:
\begin{align}
&\Delta H(\textbf{-1}) = 0 \\
&\Delta H(\textbf{+1}) = (q-2)|L_N|-|I_{N}|- h\sum_{j=0}^N|L_j|  \\
&\Delta H(\textbf{-1}^{(+)})=(q-2-h)|L_N| \\
&\Delta H(\textbf{+1}^{(-)})=|I_N|-h\sum_{j=0}^{N-1}|L_j|.
\end{align}

It is easy to see that for each value of $h>0$, we have $\Delta H(\textbf{+1})<\Delta H(\textbf{-1}^{(+)})$.
For $h>h^*_1$, we compare $\Delta H(\textbf{+1}^{(-)})$ with $\Delta H(\textbf{+1})$.

We obtain that $\Delta H(\textbf{+1}) <\Delta H(\textbf{+1}^{(-)}) $ if and only if $ (q-2-h)|L_N|-2|I_{N}|<0$, i.e., -for $h>h^*_3$, where  
\[
h^*_3:=q-2-\frac{2|I_N|}{|L_N|}.
\]
Thus, it is enough proving that $h^*_1>h_3^*$.
\begin{align}
    h^*_1-h^*_3=\frac{(q-2)|L_N|-|I_{N}|}{\sum_{j=0}^N|L_j|}- \left ( q-2-\frac{2|I_N|}{|L_N|} \right )
\end{align}
This function is decreasing in $N \geq 1$ and we conclude with
\begin{align}
    \lim_{N \to \infty} [h_1^*(N)-h^*_3(N)]=0.
\end{align}

For $h<h^*_1$, it will be enough to compare $\Delta H(\textbf{+1}^{(-)})$ with $\Delta H(\textbf{-1})$.

We obtain that $\Delta H(\textbf{-1}) <\Delta H(\textbf{+1}^{(-)}) $ if and only if $h<h_4^*$, where
\[
h^*_4:=\frac{|I_N|}{\sum_{j=0}^{N-1}|L_j|}.
\]
It remains to show that $h_1^*<h_4^*$.
\begin{align}
    h^*_1-h^*_4=\frac{(q-2)|L_N|-|I_{N;p,q}|}{\sum_{j=0}^N|L_j|}- \frac{|I_{N;p,q}|}{\sum_{j=0}^{N-1}|L_j|}
\end{align}
This function is increasing in $N \geq 1$ and we conclude with
\begin{align}
    \lim_{N \to \infty} [h_1^*(N)-h^*_4(N)]=0.
\end{align}

A schematic representation of all values $h^*$ can be found in Figure \ref{fig:valueh}.
\begin{figure}[!htb]
    \centering
        \includegraphics[width=0.7\textwidth]{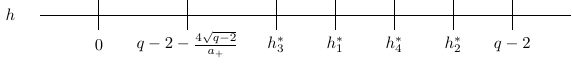} 
        \caption{A schematic representation of the values of $h$. $h_1^*, h^*_2$ are defined in Equation \eqref{eq:specifich}, $h^*_3,h^*_4$ are defined in above.}\label{fig:valueh}
\end{figure}

\end{proof}

\subsection{Region I: $h < h^*_1$ }\label{sec:region1}

In this parameter regime, a heuristic analysis suggests that the system exhibits two different behaviors. In particular, if $h \in \left (0 , q-2-\frac{4 \sqrt{q-2}}{a_+}\right )$, at a small energy cost, any configuration falls into the cycle of the stable state $\textbf{-1}$. Indeed, configurations containing non-regular clusters of pluses exhibit a low stability level, as in the previously discussed case for $h \in \left (h_1^*, h_2^* \right )$. As for configurations with regular clusters of pluses, one can observe that their energy decreases as the radius decreases. Therefore, starting from the homogeneous state $\textbf{+1}$ and flipping the minus spins from the boundary, the system evolves toward the stable state with only a small energetic cost. Indeed
\begin{align}\label{eq:energy_balls1}
    H(\sigma_{r})-H(\sigma_{r-1})&=|I_{r;p,q}|-|I_{r-1;p,q}|-h|L_{r-1}| \notag \\
    &=(q-2-h)|L_{r-1}|-2|I_{r-1;p,q}|
\end{align}
is positive if and only if
\begin{align}\label{eq:energy_balls2}
    c_{p,q}\lambda_-^{r} \left [ \left (a_+(q-2-h)-4 \sqrt{q-2} \right )\frac{\lambda_+^{r}}{\lambda_-^{r}}  +4 \sqrt{q-2} +a_-(q-2-h)\right ]>0.
\end{align}
that is true for $h<q-2-\frac{4 \sqrt{q-2}}{a_+}$. See Figure \ref{fig:regionI} for a schematic representation of the energy landscape.

\begin{figure}[!htb]
    \centering
        \includegraphics[width=0.4\textwidth]{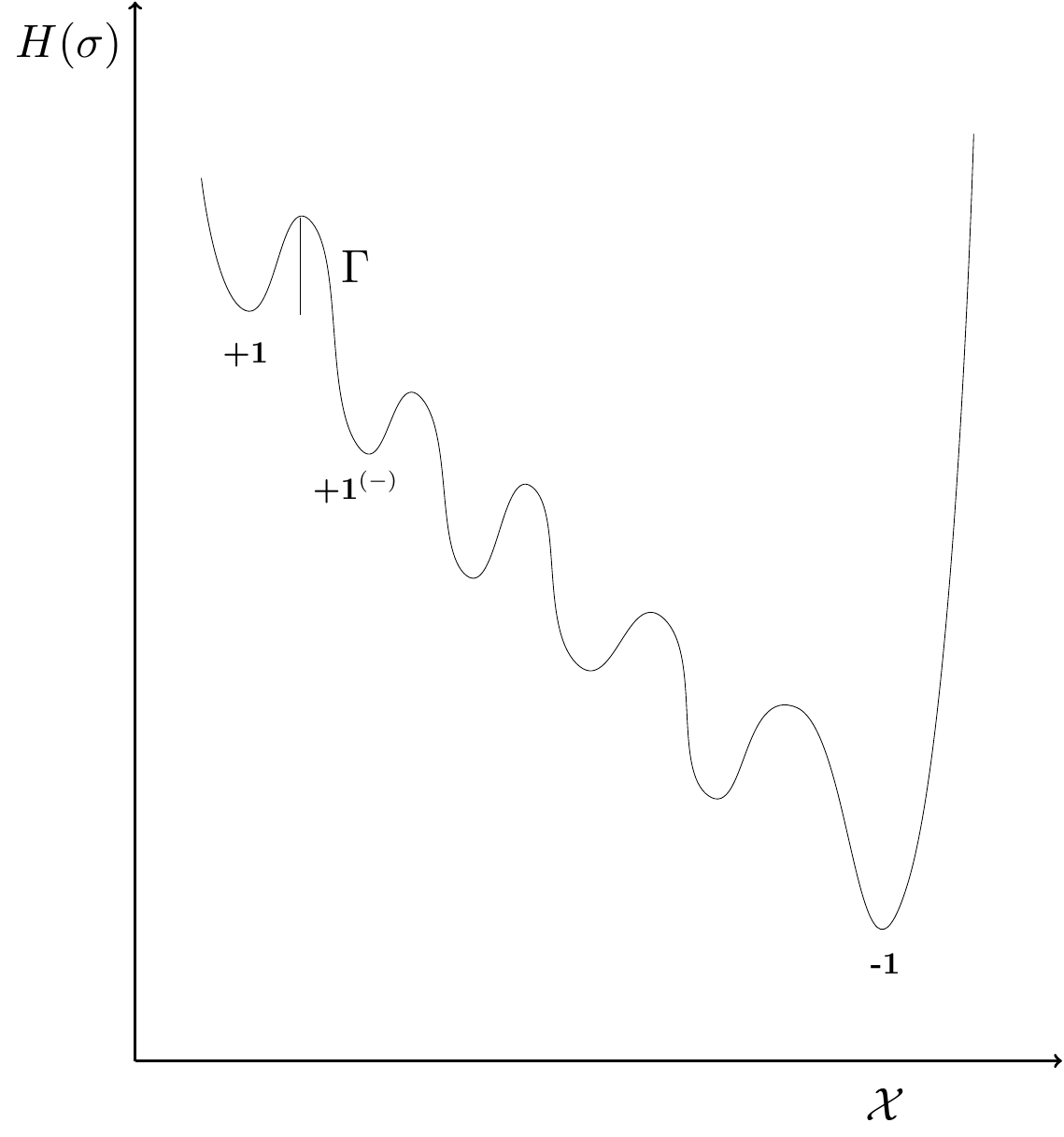} \caption{A schematic representation for the energy landscape depending on  $h \in \left (0,q-2-\frac{4 \sqrt{q-2}}{a_+} \right )$.}\label{fig:regionI}
\end{figure}

Otherwise, if we consider $h \in \left (q-2-\frac{4 \sqrt{q-2}}{a_+}, h_1^*\right )$, then the system initialized in $\textbf{+1}$ moves toward $\textbf{-1}$ through a heterogeneous nucleation, in which minus spins flip into pluses starting from the boundary and progressively shrinking the regular clusters of minuses.
We observe that in this case the energy of $\textbf{+1}^{(-)}$ is smaller (resp. greater) than the energy of $\textbf{+1}$ if $h>h^*_3$ see Figure \ref{fig:regionIbis} for a schematic representation (resp. $h<h^*_3$ and Figure \ref{fig:regionIa} for a schematic representation of the energy landscape). 

\begin{figure}[!hbtp]
    \centering
    \begin{minipage}{0.45\textwidth}
        \centering
        \includegraphics[width=1\textwidth]{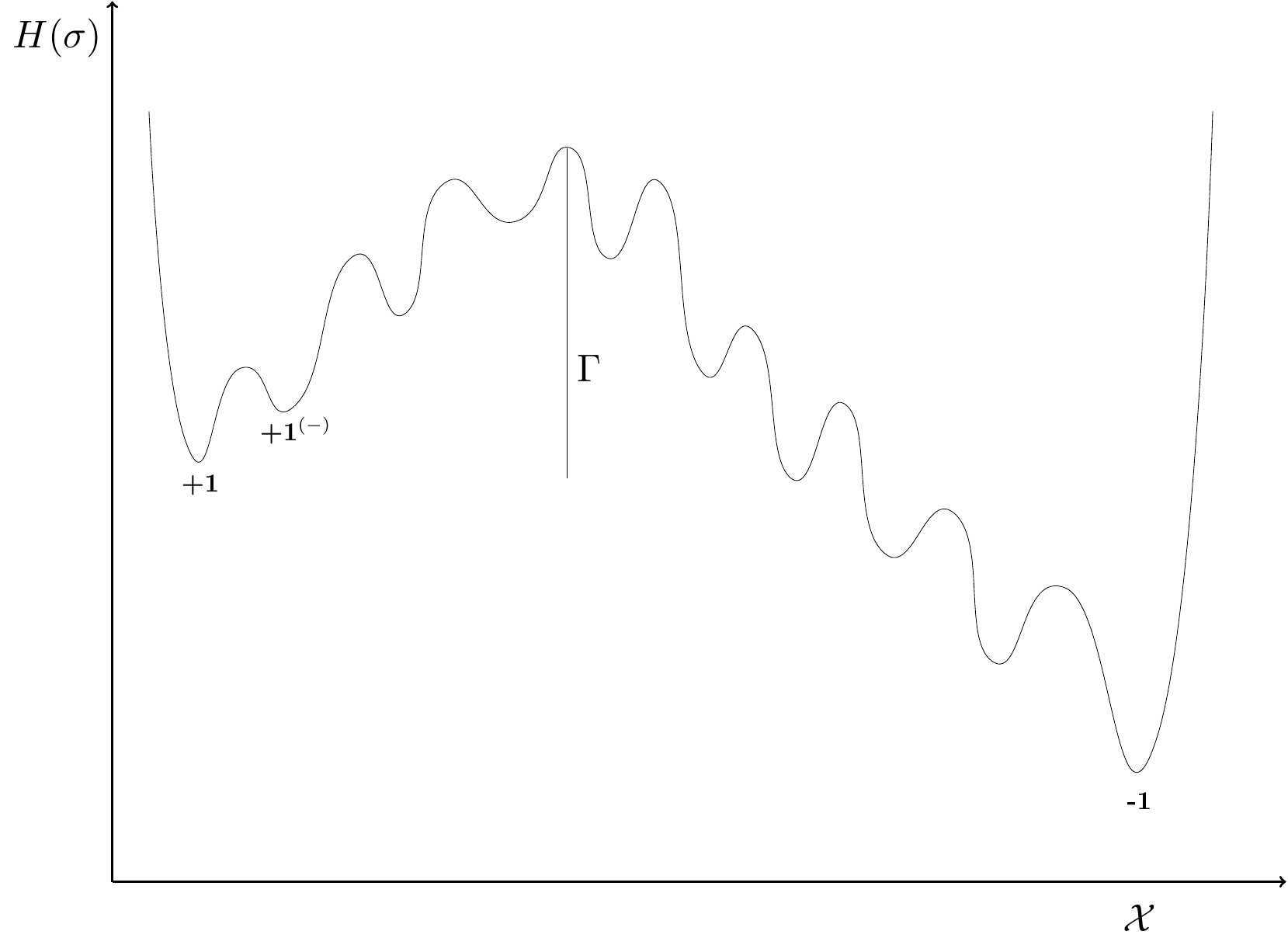}   
        \caption{Schematic representations of the energy landscape for $h \in (h^*_3, h^*_1)$ respectively.}\label{fig:regionIbis}
    \end{minipage}\hfill
    \begin{minipage}{0.45\textwidth}
        \centering    
        \includegraphics[width=1\textwidth]{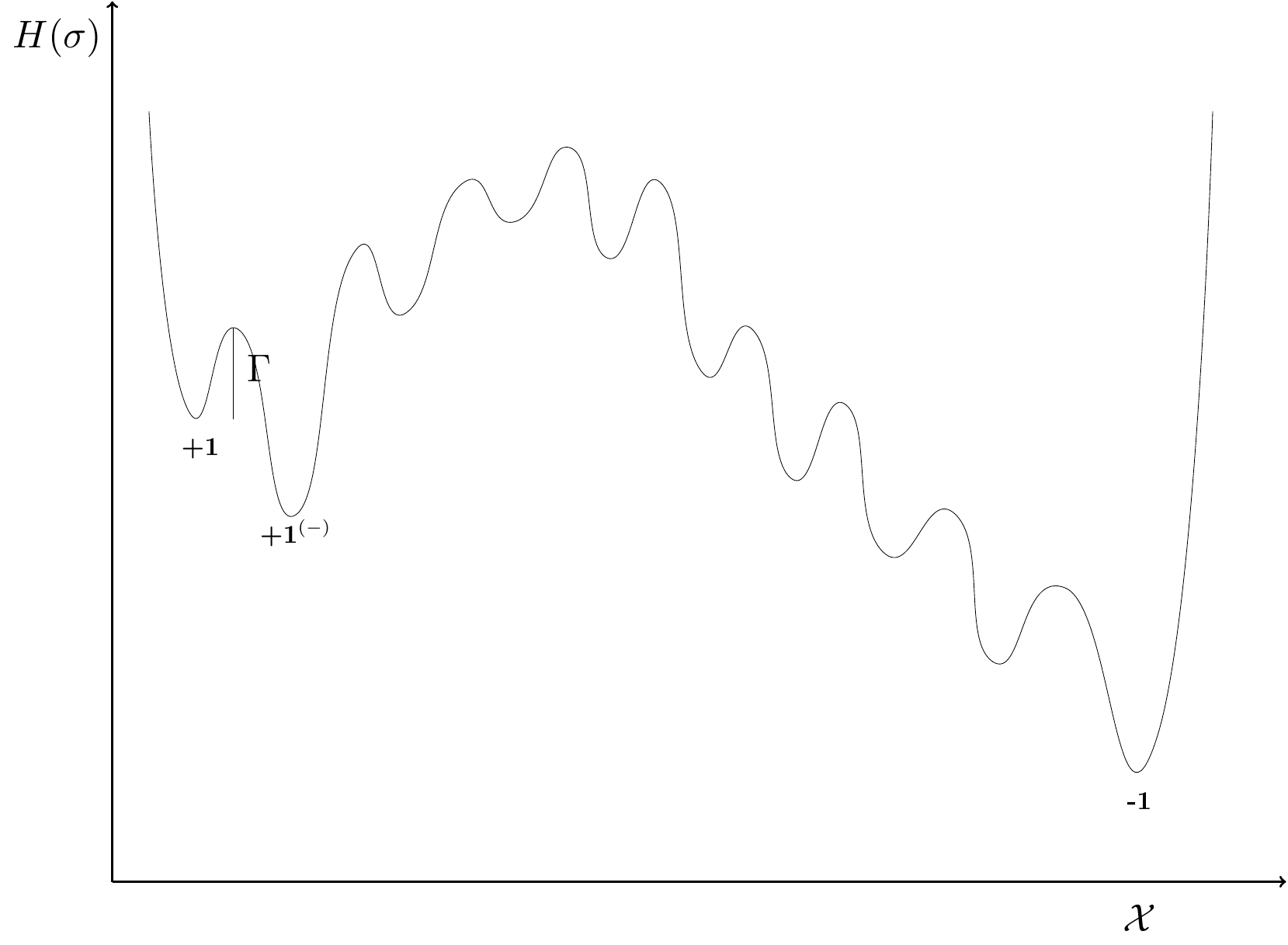} \caption{Schematic representations of the energy landscape for $h \in \left (q-2-\frac{4 \sqrt{q-2}}{a_+}, h^*_3 \right )$.}\label{fig:regionIa}     
    \end{minipage}
\end{figure}

\subsection{Region II: $h \geq h^*_2$}\label{sec:region2}

In this region, the system evolves from the homogeneous state $\textbf{-1}$ toward the stable state $\textbf{+1}$ with different behavior according to the value of $h$. 

In particular, for $h>q-2$, if a configuration contains at least a plus spin, than the system proceeds toward $\textbf{+1}$ decreasing its energy by flipping the minus spins closest to the cluster of pluses. The value of $q-h$ corresponds to the energy cost of the first swap of a minus into a plus in the state $\textbf{-1}$. 
For an example of a schematic representation of the energy landscape see Figure \ref{fig:regionII}.

Otherwise, if $h \in \left [ h^*_2, q-2 \right ]$, the system starting from $\textbf{-1}$ reaches $\textbf{+1}$ by creating consecutive regular cluster of pluses as in case $h \in (h^*_1, h^*_2)$. However in this case the energy of a regular cluster of pluses decreases as its radius, indeed
recalling \ref{eq:energy_balls1}, it is negative for all value of $h>h^*_2$. See Figure \ref{fig:regionIIa} for a schematic representation of the energy landscape.

\begin{figure}[!hbtp]
    \centering
    \begin{minipage}{0.45\textwidth}
        \centering
       \includegraphics[width=0.7\textwidth]{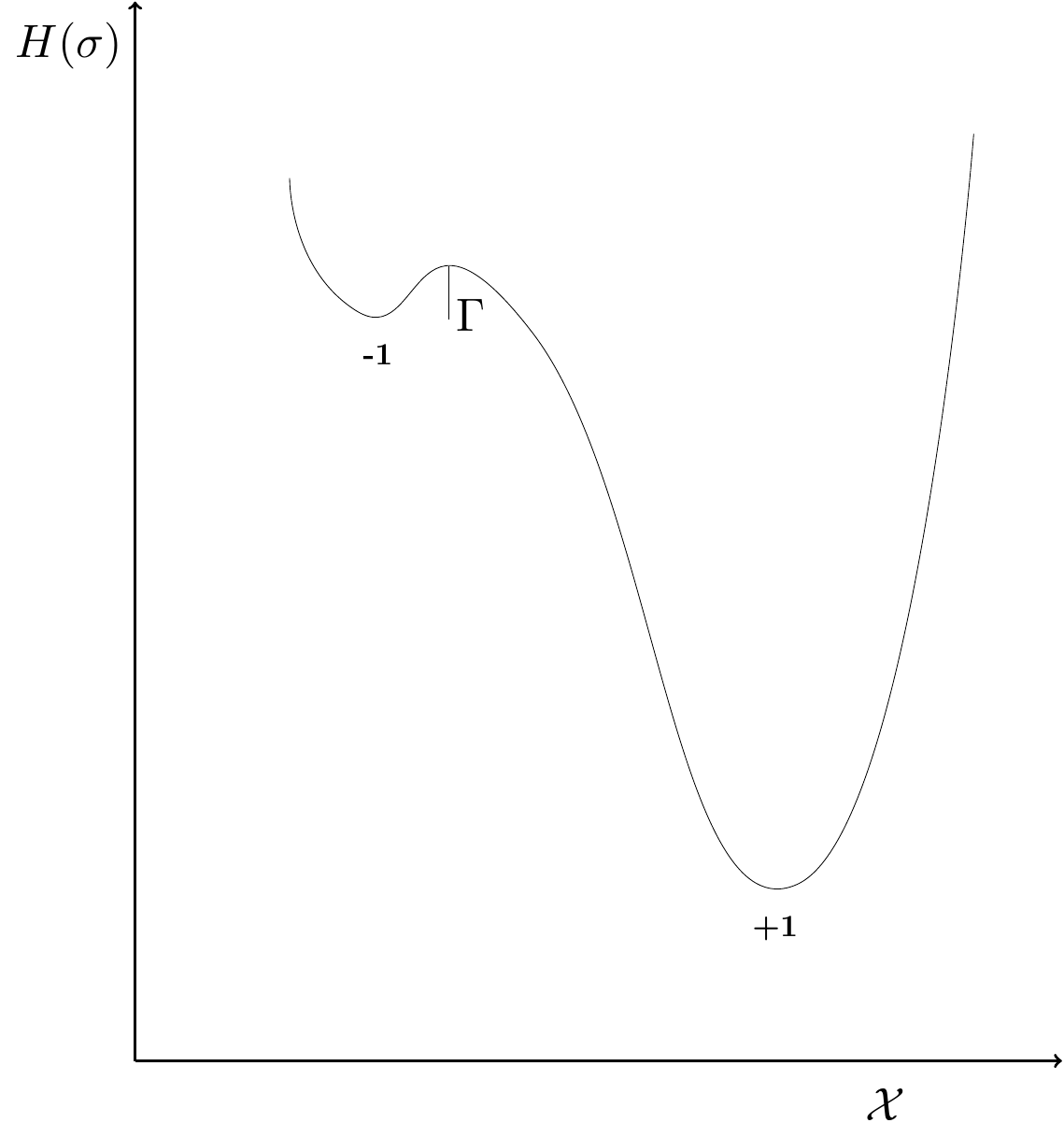}  
        \caption{A schematic representation of the energy landscape for $h>q-2$.}\label{fig:regionII}    
    \end{minipage}\hfill
    \begin{minipage}{0.45\textwidth}
        \centering    
            \includegraphics[width=0.8\textwidth]{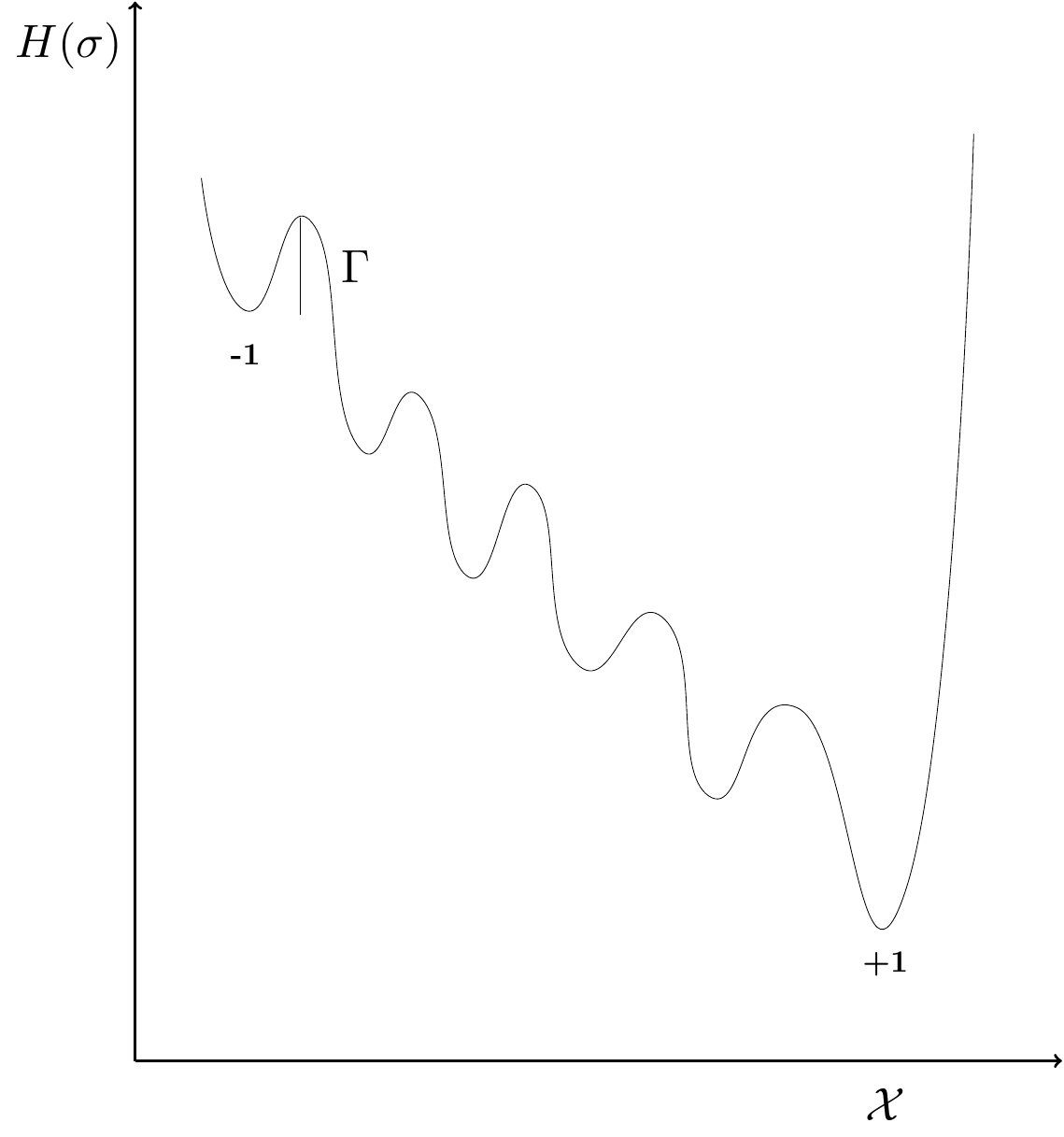} 
        \caption{A schematic representation of the energy landscape for $h \in \left [ h^*_2, q-2 \right ]$.}\label{fig:regionIIa}    
    \end{minipage}
\end{figure}

\section{Appendix}\label{sec:appendix}
In this Section we show two different examples of critical configurations for the same values of the parameters $p,q=5$ and different external magnetic fields. Note that in both cases the droplets are arranged in such a way that they minimize the perimeter, see Definition \ref{def:minset}.

Fixed $N$ large enough (for example $N>20$), then we can determine numerically the parameter range of $h$,
$(h_1^*,h_2^*)=(2.2361 , 2.25)$  with a width of at most $10^{-4}$.

\paragraph{Example 1.}
We choose $h=2.24$. In this case the critical radius is equal to $r^*=1$ with strip of length 13.

The critical configuration is a ball $B_{1;5,5}$ with a strip of length $k=13$. See the energy landscape in the first panel of Figure \ref{fig:examples} and Figure \ref{fig:example_zoom} for a zoom of the representation. 

In the first panel of Figure \ref{fig:examples}, the first (resp. the last) energy value corresponds to the energy of a configuration containing a cluster of pluses with shape $B_{1;5,5}$ (resp. $B_{2;5,5}$) in a sea of minuses. We observe that the first energy is greater than the second one, 
since $r^*=1$.

\paragraph{Example 2.}
We choose $h=2.2364$. In this case the critical radius is equal to $r^*=1$ with strip of length 55.

The critical configuration is a ball $B_{1;5,5}$ with a strip of length $k=55$. See the energy landscape in the second panel of Figure \ref{fig:examples}and Figure \ref{fig:example_zoom} for a zoom of the representation. 

In the second panel of Figure \ref{fig:examples}, the first (resp. the last) energy value corresponds to the energy of a configuration containing a cluster of pluses with shape $B_{1;5,5}$ (resp. $B_{2;5,5}$) in a sea of minuses. We observe that the first energy is greater than the second one, 
since $r^*=1$.

\begin{figure}[htb]
\centering
\includegraphics[scale=0.5]{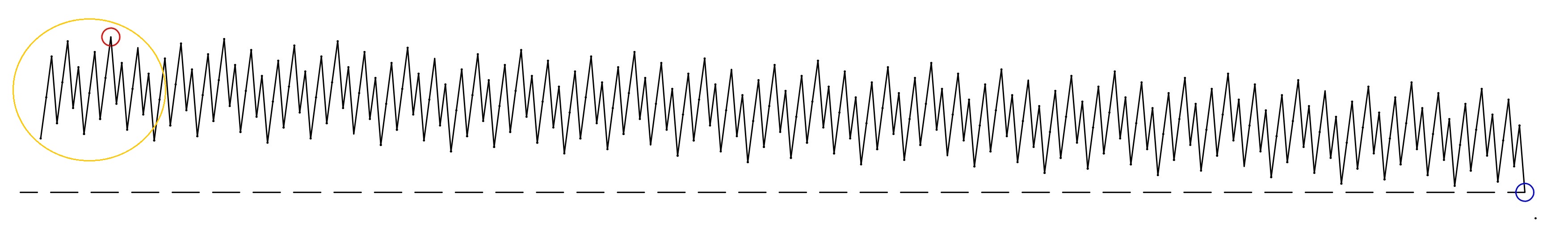} \,\,\,\,\, \includegraphics[scale=0.5]{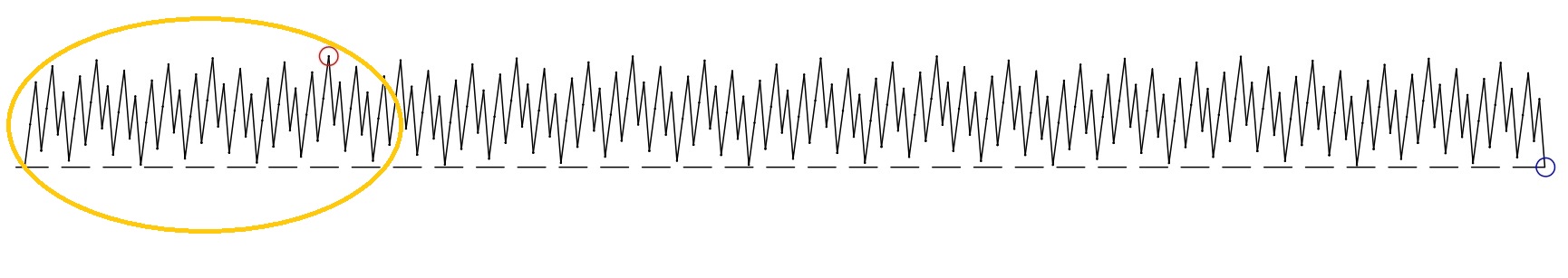}
\caption{The energy landscape along the reference path between the configurations containing $B_{1;5,5}$ and $B_{2;5,5}$, for the two examples described in the Appendix. In red (resp. blue) the maximal (minimal) energy along this path.} \label{fig:examples}
\end{figure}
\begin{figure}[htb]
\centering \includegraphics[scale=0.5]{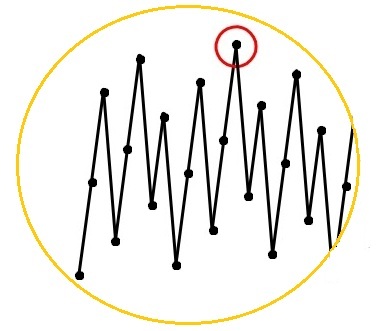} \,\,\,\,\,\,\,\,\,\, \includegraphics[scale=0.9]{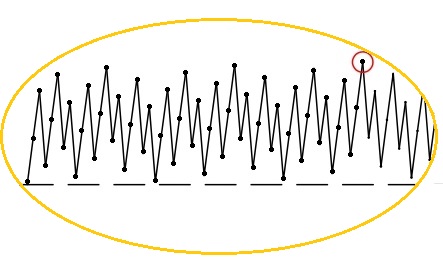}
\caption{On the left, the first part of the energy landscape of Example 1. Starting from the configuration containing a ball of pluses $B_{1;5,5}$, we flip the minuses adjacent to the ball in the order described in Section \ref{sec:recurrence}. Specifically, the energy increases by $q-2-h=3-h$ for $p-3=2$ times, then it decreases by $q-4-h=1-h$ and so on. After $5p-12=13$ steps, the energy is increased by $9(3-h)$ and it is decreased by $4(1-h)$. Thus, the energy of the critical configuration is obtained by summing the energy of the ball $B_{1;5,5}$ and is equal to  $31-13h$. \\
On the right, the first part of the energy landscape of Example 2. After $21p-50=55$ steps, the energy is increased by $35(3-h)$ and it is decreased by $20(1-h)$. Thus, the energy of the critical configuration is obtained by summing the energy of the ball $B_{1;5,5}$ and is equal to $125-55h$.} \label{fig:example_zoom}
\end{figure}

\section*{Funding}
V.J. and W.M.R. were supported by the NWO (Dutch Research Organization) grant  VI.Vidi.213.112.

\section*{Conflicts of interest}
The authors declare no conflicts of interest regarding this manuscript.

\section*{Data availability statement}
We do not analyse or generate any datasets, because our work proceeds within a theoretical and mathematical approach.

\section*{Acknowledgments} 
The authors wish to thank engineer E. Panichi for his valuable support in the use of CAD for the preparation of the figures included in the appendix.
V.J. thanks to GNAMPA. 

\bibliographystyle{plain}  
\bibliography{biblio}

\begin{thebibliography}{10}

\bibitem{aizenman1980}
Michael Aizenman.
\newblock {Translation invariance and instability of phase coexistence in the two dimensional Ising system}.
\newblock {\em Communications in Mathematical Physics}, 73(1):83--94, 1980.

\bibitem{apollonio2022metastability}
Valentina Apollonio, Vanessa Jacquier, Francesca~Romana Nardi, and Alessio Troiani.
\newblock Metastability for the ising model on the hexagonal lattice.
\newblock {\em Electronic Journal of Probability}, 27:1--48, 2022.

\bibitem{baldassarri2023metastability}
Simone Baldassarri and Vanessa Jacquier.
\newblock Metastability for kawasaki dynamics on the hexagonal lattice.
\newblock {\em Journal of Statistical Physics}, 190(3):46, 2023.

\bibitem{bashiri2017note}
K~Bashiri.
\newblock On the metastability in three modifications of the ising model.
\newblock {\em Markov processes and related fields}, 25(3):483--532, 2019.

\bibitem{beltran2012tunneling}
J~Beltr{\'a}n and C~Landim.
\newblock Tunneling and metastability of continuous time {Markov} chains ii, the nonreversible case.
\newblock {\em Journal of Statistical Physics}, 149(4):598--618, 2012.

\bibitem{beltran2010tunneling}
Johel Beltran and Claudio Landim.
\newblock Tunneling and metastability of continuous time {Markov} chains.
\newblock {\em Journal of Statistical Physics}, 140(6):1065--1114, 2010.

\bibitem{beltran2015martingale}
Johel Beltr{\'a}n and Claudio Landim.
\newblock A martingale approach to metastability.
\newblock {\em Probability Theory and Related Fields}, 161:267--307, 2015.

\bibitem{berger2005glauber}
Noam Berger, Claire Kenyon, Elchanan Mossel, and Yuval Peres.
\newblock Glauber dynamics on trees and hyperbolic graphs.
\newblock {\em Probability Theory and Related Fields}, 131(3):311--340, 2005.

\bibitem{bet2021critical}
Gianmarco Bet, Anna Gallo, and Francesca~R Nardi.
\newblock Critical configurations and tube of typical trajectories for the potts and ising models with zero external field.
\newblock {\em Journal of Statistical Physics}, 184(3):30, 2021.

\bibitem{bet2022metastability}
Gianmarco Bet, Anna Gallo, and Francesca~R Nardi.
\newblock Metastability for the degenerate potts model with negative external magnetic field under glauber dynamics.
\newblock {\em Journal of Mathematical Physics}, 63(12), 2022.

\bibitem{bet2024metastability}
Gianmarco Bet, Anna Gallo, and Francesca~R Nardi.
\newblock Metastability for the degenerate potts model with positive external magnetic field under glauber dynamics.
\newblock {\em Stochastic Processes and their Applications}, 172:104343, 2024.

\bibitem{bet2020effect}
Gianmarco Bet, Vanessa Jacquier, and Francesca~R Nardi.
\newblock Effect of energy degeneracy on the transition time for a series of metastable states: application to probabilistic cellular automata.
\newblock {\em Journal Statistical Physics}, 184(8), 2021.

\bibitem{bianchi2016metastable}
Alessandra Bianchi and Alexandre Gaudilliere.
\newblock Metastable states, quasi-stationary distributions and soft measures.
\newblock {\em Stochastic Processes and their Applications}, 126(6):1622--1680, 2016.

\bibitem{bianchi2020soft}
Alessandra Bianchi, Alexandre Gaudilli{\`e}re, and Paolo Milanesi.
\newblock On soft capacities, quasi-stationary distributions and the pathwise approach to metastability.
\newblock {\em Journal of Statistical Physics}, 181(3):1052--1086, 2020.

\bibitem{osti_1979736}
Igor Boettcher, Alexey~V. Gorshkov, Alicia~J. Kollár, Joseph Maciejko, Steven Rayan, and Ronny Thomale.
\newblock Crystallography of hyperbolic lattices.
\newblock {\em Physical Review. B}, 105(12), 03 2022.

\bibitem{bovier2016metastability}
Anton Bovier and Frank {Den}~Hollander.
\newblock {\em Metastability: a potential-theoretic approach}, volume 351.
\newblock Springer, 2016.

\bibitem{bovier2010homogeneous}
Anton Bovier, Frank {Den}~Hollander, Cristian Spitoni, et~al.
\newblock Homogeneous nucleation for {Glauber} and {Kawasaki} dynamics in large volumes at low temperatures.
\newblock {\em The Annals of Probability}, 38(2):661--713, 2010.

\bibitem{bovier2002metastability}
Anton Bovier, Michael Eckhoff, V{\'e}ronique Gayrard, and Markus Klein.
\newblock Metastability and low lying spectra in reversible {Markov} chains.
\newblock {\em Communications in mathematical physics}, 228(2):219--255, 2002.

\bibitem{bovier2006sharp}
Anton Bovier, F~den Hollander, and FR2218873 Nardi.
\newblock Sharp asymptotics for kawasaki dynamics on a finite box with open boundary.
\newblock {\em Probability theory and related fields}, 135:265--310, 2006.

\bibitem{cassandro1984metastable}
Marzio Cassandro, Antonio Galves, Enzo Olivieri, and Maria~Eul{\'a}lia Vares.
\newblock Metastable behavior of stochastic dynamics: a pathwise approach.
\newblock {\em Journal of Statistical Physics}, 35(5-6):603--634, 1984.

\bibitem{cerf2013nucleation}
Rapha{\"e}l Cerf and Francesco Manzo.
\newblock Nucleation and growth for the {Ising} model in $ d $ dimensions at very low temperatures.
\newblock {\em The Annals of Probability}, 41(6):3697--3785, 2013.

\bibitem{chen23}
Anffany Chen, Hauke Brand, Tobias Helbig, Tobias Hofmann, Stefan Imhof, Alexander Fritzsche, Tobias Kie{\ss}ling, Alexander Stegmaier, Lavi~K Upreti, Titus Neupert, et~al.
\newblock Hyperbolic matter in electrical circuits with tunable complex phases.
\newblock {\em Nature Communications}, 14(1):622, 2023.

\bibitem{cirillo2013relaxation}
Emilio N.~M. Cirillo and Francesca~R. Nardi.
\newblock Relaxation height in energy landscapes: an application to multiple metastable states.
\newblock {\em Journal of Statistical Physics}, 150(6):1080--1114, 2013.

\bibitem{cirillo2015metastability}
Emilio N.~M. Cirillo, Francesca~R. Nardi, and Julien Sohier.
\newblock Metastability for general dynamics with rare transitions: escape time and critical configurations.
\newblock {\em Journal of Statistical Physics}, 161(2):365--403, 2015.

\bibitem{cirillo2008competitive}
Emilio N.~M. Cirillo, Francesca~R. Nardi, and Cristian Spitoni.
\newblock Competitive nucleation in reversible {Probabilistic} {Cellular} {Automata}.
\newblock {\em Physical Review E}, 78(4):040601, 2008.

\bibitem{cirillo2017sum}
Emilio N.~M. Cirillo, Francesca~R. Nardi, and Cristian Spitoni.
\newblock Sum of exit times in a series of two metastable states.
\newblock {\em The European Physical Journal Special Topics}, 226(10):2421--2438, 2017.

\bibitem{cirillo2022metastability}
Emilio Nicola~Maria Cirillo, Vanessa Jacquier, and Cristian Spitoni.
\newblock Metastability of synchronous and asynchronous dynamics.
\newblock {\em Entropy}, 24(4):450, 2022.

\bibitem{cirillo2024homogeneous}
Emilio~NM Cirillo, Vanessa Jacquier, and Cristian Spitoni.
\newblock Homogeneous and heterogeneous nucleation in the three-state blume--capel model.
\newblock {\em Physica D: Nonlinear Phenomena}, 461:134125, 2024.

\bibitem{cirillo1998metastability}
Emilio~NM Cirillo and Joel~L Lebowitz.
\newblock Metastability in the two-dimensional ising model with free boundary conditions.
\newblock {\em Journal of Statistical Physics}, 90:211--226, 1998.

\bibitem{cirillo2003metastability}
Emilio~NM Cirillo and Francesca~R Nardi.
\newblock Metastability for a stochastic dynamics with a parallel heat bath updating rule.
\newblock {\em Journal of statistical physics}, 110(1):183--217, 2003.

\bibitem{CNrelax13}
Emilio~NM Cirillo and Francesca~R Nardi.
\newblock Relaxation height in energy landscapes: an application to multiple metastable states.
\newblock {\em Journal of Statistical Physics}, 150(6):1080--1114, 2013.

\bibitem{cirillo2008metastability}
Emilio~NM Cirillo, Francesca~R Nardi, and Cristian Spitoni.
\newblock Metastability for reversible probabilistic cellular automata with self-interaction.
\newblock {\em Journal of Statistical Physics}, 132(3):431--471, 2008.

\bibitem{DCLN}
Matteo D'Achille, Loren Coquille, and Arnaud {Le Ny}.
\newblock {Extremal Ising Gibbs States on Hyperbolic Lattices}.
\newblock {\em arXiv:2504.19553}, 2025.

\bibitem{d2025minimal}
Matteo d'Achille, Vanessa Jacquier, and Wioletta~M Ruszel.
\newblock On minimal shapes and topological invariants in hyperbolic lattices.
\newblock {\em arXiv:2504.14080}, 2025.

\bibitem{dehghanpour1997metropolis}
Pouria Dehghanpour and Roberto~H Schonmann.
\newblock Metropolis dynamics relaxation via nucleation and growth.
\newblock {\em Communications in mathematical physics}, 188(1):89--119, 1997.

\bibitem{den2011kawasaki}
F~Den~Hollander, Francesca~R Nardi, and Alessio Troiani.
\newblock Kawasaki dynamics with two types of particles: stable/metastable configurations and communication heights.
\newblock {\em Journal of Statistical Physics}, 145:1423--1457, 2011.

\bibitem{den2012metastability}
Frank den Hollander, Francesca Nardi, and Alessio Troiani.
\newblock Metastability for kawasaki dynamics at low temperature with two types of particles.
\newblock {\em Electronic Journal of Probability}, 17:1--26, 2012.

\bibitem{gaudilliere2009ideal}
Alexandre Gaudilli{\`e}re, Frank Den~Hollander, Francesca~R Nardi, Enzo Olivieri, and Elisabetta Scoppola.
\newblock Ideal gas approximation for a two-dimensional rarefied gas under kawasaki dynamics.
\newblock {\em Stochastic Processes and their Applications}, 119(3):737--774, 2009.

\bibitem{gaudillierelandim2014}
Alexandre Gaudilliere and Claudio Landim.
\newblock {A Dirichlet principle for non reversible {Markov} chains and some recurrence theorems}.
\newblock {\em {Probability Theory and Related Fields}}, 158:55--89, 2014.

\bibitem{gaudilliere2020asymptotic}
Alexandre Gaudilli{\`e}re, Paolo Milanesi, and Maria~Eul{\'a}lia Vares.
\newblock Asymptotic exponential law for the transition time to equilibrium of the metastable kinetic {Ising} model with vanishing magnetic field.
\newblock {\em Journal of Statistical Physics}, pages 1--46, 2020.

\bibitem{gaudilliere2010upper}
Alexandre Gaudilli{\`e}re and Francesca~R Nardi.
\newblock An upper bound for front propagation velocities inside moving populations.
\newblock {\em Brazilian Journal of Probability and Statistics}, 24(2):256--278, 2010.

\bibitem{haggstrom2002explicit}
Olle H{\"a}ggstr{\"o}m, Johan Jonasson, and Russell Lyons.
\newblock Explicit isoperimetric constants and phase transitions in the random-cluster model.
\newblock {\em The Annals of Probability}, 30(1):443--473, 2002.

\bibitem{higuchi1979}
Y~Higuchi.
\newblock {On the absence of non translation invariant Gibbs states for the two dimensional Ising model}.
\newblock In {\em Colloquia Math. Sociatatis Janos Bolyai}, volume~27, pages 517--534. Random fields, 1979.

\bibitem{hollander2000metastability}
F~den Hollander, Enzo Olivieri, and Elisabetta Scoppola.
\newblock Metastability and nucleation for conservative dynamics.
\newblock {\em Journal of Mathematical Physics}, 41(3):1424--1498, 2000.

\bibitem{jacquier2025exploring}
Vanessa Jacquier.
\newblock Exploring metastability in ising models: critical droplets, energy barriers and exit time.
\newblock {\em Mathematical Physics, Analysis and Geometry}, 28(3):18, 2025.

\bibitem{jovanovski2017metastability}
Oliver Jovanovski.
\newblock Metastability for the {Ising} model on the hypercube.
\newblock {\em Journal of Statistical Physics}, 167(1):135--159, 2017.

\bibitem{electro}
A.J. Koll\'ar, M.~Fitzpatrick, and A.A Houck.
\newblock Hyperbolic lattices in circuit quantum electrodynamics.
\newblock {\em Nature}, 571:45--50, 2019.

\bibitem{kotecky1993droplet}
Roman Kotecky` and Enzo Olivieri.
\newblock Droplet dynamics for asymmetric {Ising} model.
\newblock {\em Journal of statistical physics}, 70(5):1121--1148, 1993.

\bibitem{manzo1998relaxation}
F~Manzo and E~Olivieri.
\newblock Relaxation patterns for competing metastable states: a nucleation and growth model.
\newblock In {\em {Markov} Proc. Relat. Fields}, volume~4, pages 549--570, 1998.

\bibitem{manzo2001dynamical}
F~Manzo and E~Olivieri.
\newblock Dynamical {Blume}--{Capel} model: competing metastable states at infinite volume.
\newblock {\em Journal of Statistical Physics}, 104(5-6):1029--1090, 2001.

\bibitem{manzo2004essential}
Francesco Manzo, Francesca~R. Nardi, Enzo Olivieri, and Elisabetta Scoppola.
\newblock On the essential features of metastability: tunnelling time and critical configurations.
\newblock {\em Journal of Statistical Physics}, 115(1-2):591--642, 2004.

\bibitem{MNOS04}
Francesco Manzo, Francesca~R Nardi, Enzo Olivieri, and Elisabetta Scoppola.
\newblock On the essential features of metastability: tunnelling time and critical configurations.
\newblock {\em Journal of Statistical Physics}, 115(1-2):591--642, 2004.

\bibitem{martinelli2004glauber}
Fabio Martinelli, Alistair Sinclair, and Dror Weitz.
\newblock Glauber dynamics on trees: boundary conditions and mixing time.
\newblock {\em Communications in Mathematical Physics}, 250(2):301--334, 2004.

\bibitem{monroe}
J.L. Monroe.
\newblock {Comment on: Ising models on hyperbolic graphs}.
\newblock {\em Journal Statistical Physics}, 88:513--518, 1997.

\bibitem{nardi2012sharp}
FR~Nardi and C~Spitoni.
\newblock Sharp asymptotics for stochastic dynamics with parallel updating rule.
\newblock {\em Journal of statistical physics}, 146(4):701--718, 2012.

\bibitem{nardi1996low}
Francesca~R. Nardi and Enzo Olivieri.
\newblock Low temperature stochastic dynamics for an {Ising} model with alternating field.
\newblock In {\em {Markov} Proc. Relat. Fields}, volume~2, pages 117--166, 1996.

\bibitem{nardi2016hitting}
Francesca~R Nardi, Alessandro Zocca, and Sem~C Borst.
\newblock Hitting time asymptotics for hard-core interactions on grids.
\newblock {\em Journal of Statistical Physics}, 162(2):522--576, 2016.

\bibitem{neves1992behavior}
E~Jord{\~a}o Neves and Roberto~H Schonmann.
\newblock Behavior of droplets for a class of {Glauber} dynamics at very low temperature.
\newblock {\em Probability theory and related fields}, 91(3-4):331--354, 1992.

\bibitem{olivieri2005large}
Enzo Olivieri and Maria~Eul{\'a}lia Vares.
\newblock {\em Large deviations and metastability}, volume 100.
\newblock Cambridge University Press, 2005.

\bibitem{RNO}
Ronald Rietman, Bernard Nienhuis, and Jaan Oitmaa.
\newblock {The Ising model on hyperlattices}.
\newblock {\em Journal of Physics A: Mathematical and General}, 25(24):6577, 1992.

\bibitem{schonmann1994slow}
Roberto~H Schonmann.
\newblock Slow droplet-driven relaxation of stochastic {Ising} models in the vicinity of the phase coexistence region.
\newblock {\em Communications in Mathematical Physics}, 161(1):1--49, 1994.

\bibitem{schonmann1998wulff}
Roberto~H Schonmann and Senya~B Shlosman.
\newblock Wulff droplets and the metastable relaxation of kinetic {Ising} models.
\newblock {\em Communications in mathematical physics}, 194(2):389--462, 1998.

\bibitem{wu}
C.C. Wu.
\newblock {Ising models on hyperbolic graphs}.
\newblock {\em Journal Statistical Physics}, 85:251--259, 1996.

\bibitem{cryst}
Sunkyu Yu, Xianji Piao, and Namkyoo Park.
\newblock {Topological Hyperbolic Lattices}.
\newblock {\em Phys. Rev. Lett.}, 125:053901, Jul 2020.

\end{thebibliography}

\end{document}